\documentclass[12pt,letterpaper]{article}
\usepackage{amsmath,amsfonts,amsthm,amssymb}

\swapnumbers 

\newtheorem{lemma}{Lemma}[section]
\newtheorem{corollary}[lemma]{Corollary}
\newtheorem{theorem}[lemma]{Theorem}
\newtheorem{assumptions}[lemma]{Standing Assumptions}
\newtheorem{assumption}[lemma]{Standing Assumption}

\theoremstyle{definition} 
\newtheorem{definition}[lemma]{Definition}
\newtheorem{remark}[lemma]{Remark}
\newtheorem{example}[lemma]{Example}
\newtheorem{remarks}[lemma]{Remarks}

\newcommand\rationals{{\mathbb Q}}
\newcommand\reals{{\mathbb R}}

\newcommand{\dg}{\sp{\text{\rm o}}}

\begin{document}

\title{A projection and an effect in a synaptic algebra}

\author{David J. Foulis{\footnote{Emeritus Professor, Department of
Mathematics and Statistics, University of Massachusetts, Amherst,
MA; Postal Address: 1 Sutton Court, Amherst, MA 01002, USA;
foulis@math.umass.edu.}}\hspace{.05 in}, Anna Jen\v cov\'a  and Sylvia
Pulmannov\'{a}{\footnote{ Mathematical Institute, Slovak Academy of
Sciences, \v Stef\'anikova 49, SK-814 73 Bratislava, Slovakia;
pulmann@mat.savba.sk. The second and third authors were supported by
Research and Development Support Agency under the contract No.
APVV-0178-11 and grant VEGA 2/0059/12.}}}

\date{}

\maketitle

\begin{abstract}
We study a pair $p,e$ consisting of a projection $p$ (an idempotent)
and an effect $e$ (an element between $0$ and $1$) in a synaptic
algebra (a generalization of the self-adjoint part of a von Neumann
algebra). We show that some of Halmos's theory of two projections (or
two subspaces), including a version of his CS-decomposition theorem,
applies in this setting, and we introduce and study two candidates
for a commutator projection for $p$ and $e$.
\end{abstract}

{\small \noindent \emph{Mathematics subject Classification.} 17C65, 81P10, 47B15

}

\section{Introduction}

In \cite{Halmos}, P. Halmos studied two projection operators $P$ and $Q$
on a Hilbert space and proved a basic theorem, now called the
\emph{CS-decomposition theorem}, that expresses $Q$ in terms of $P$ and
positive contraction operators $C$ and $S$, called the \emph{cosine}
and the \emph{sine} operators, respectively, for $Q$ with respect to
$P$. For a lucid and extended exposition of Halmos's theory of two projections,
see \cite{Guide}. In \cite{FJP2proj}, we proved a generalization of the
CS-decomposition theorem in the setting of a so-called synaptic algebra
\cite[Theorem 5.6]{FJP2proj}.

In what follows, $A$ is a synaptic algebra with enveloping algebra $R\supseteq
A$ \cite{FSynap, FPSynap, TDSA, SymSA, ComSA, PuNote}, $P$ is the orthomodular
lattice \cite{Beran, Kalm} of projections in $A$, and $E$ is the convex effect
algebra \cite{FBEA, GPBB} of all effects in $A$. To help fix ideas, we note
that the self-adjoint part of a von Neumann algebra, and more generally of an
AW$\sp{\ast}$-algebra, forms a synaptic algebra. Numerous additional examples
are given in the literature cited above.

In this article we generalize the CS-decomposition theorem for two projections
$p,q\in P\subseteq A$ to the case of a projection $p\in P$ and an effect
$e\in E$ (Theorem \ref{th:CBSdecomp} below), and we investigate two candidates
for the commutator projection for the pair $p$ and $e$ (Section \ref
{sc:Commutators} below).

In our generalization of the CS-decomposition theorem, which we call the
\emph{CBS-decomposition theorem}, the cosine and sine effects $c$ and
$s$ introduced in \cite[Definition 4.2]{FJP2proj} are generalized
(Definition 3.1 below) and supplemented by a third effect $b$ (Definition
\ref{df:b} below).

Part of our motivation for the work in this article derives from our
interest in \emph{the infimum problem} as applied to the synaptic
algebra $A$, i.e., the problem of determining just when two effects
$e,f\in E$ have an infimum $e\wedge f$ in $E$, and if possible, finding
a perspicuous formula for $e\wedge f$ when it does exist. That
this problem is non-trivial is indicated by a remark of P. Lahti and M.
M\c{a}czynski in \cite[p. 1674]{LM} that the partial order structure of
$E$ is ``rather wild." The development in \cite{GGJ} and \cite{MG}
suggests that it might be possible to make progress on the infimum problem
for $A$ if the problem can be solved for the pair $p,e$ with $p\in P$ and
$e\in E$. We hope that our results in this article will cast some light
on the latter problem. In Section \ref{sc:Appl} below, we illustrate the
utility of the CBS-decomposition theorem by applying it to generalize
a result of T. Moreland and S. Gudder concerning the infimum problem
\cite{MG} to the setting of a synaptic algebra.

\section{Some basic definitions, notation, and facts} \label{sc:Basic}

In this section we briefly outline some notions that we shall need
below. For the definition of a synaptic algebra and more details, see
the literature cited above, especially \cite{FSynap} and \cite{SymSA}.
In what follows, the notation $:=$ means `equals by definition,' the
ordered field of real numbers and its subfield of rational numbers are
denoted by $\reals$ and $\rationals$, and `iff' abbreviates `if and only
if.'

The enveloping algebra $R$ of $A$ is a real linear associative algebra
and if $a,b\in A$, it is understood that the product $ab$, which may
or may not belong to $A$, is calculated in $R$. However, if $a$
commutes with $b$, in symbols $aCb$, then $ab=ba\in A$. The
\emph{commutant} and \emph{bicommutant} of $a$ are defined and denoted
by
\[
C(a):=\{b\in A:aCb\}\text{\ and\ }CC(a):=\{c\in A:c\in C(b)\text
{\ for all\ }b\in C(a)\},
\]
respectively. There is a \emph{unity element} $1\in A$ such that
$1a=a1=a$ for all $a\in A$.

As a subset of $R$, the synaptic algebra $A$ forms a real linear space
which is partially ordered by $\leq$ and for which $1$ is a (strong)
order unit. If $a,b\in A$ and $a\leq b$, we say that $b$ \emph{dominates}
$a$, or equivalently, that $a$ is a \emph{subelement} of $b$.

If $a,b,c\in A$, then $ab+ba,\, abc+cba\in A$. Also $aba\in A$ and
the \emph{quadratic mapping} $b\mapsto aba$ is linear and order
preserving on $A$.

If $0\leq a\in A$, there exists a unique \emph{square root}, denoted
$a\sp{1/2}\in A$ such that $0\leq a\sp{1/2}$ and $(a\sp{1/2})\sp{2}=a$;
moreover $a\sp{1/2}\in CC(a)$. Thus, if $0\leq a$, then $C(a)=
C(a\sp{2})=C(a\sp{1/2})$. If $a\in A$, then $0\leq a\sp{2}$, and
the \emph{absolute value} of $a$ is denoted and defined by $|a|:=
(a\sp{2})\sp{1/2}$. We note that $|a|\in CC(a)$. The \emph{positive
part} of $a$ is denoted and defined by $a\sp{+}:=\frac{1}{2}(|a|+a)$.
Clearly, $a\sp{+}\in CC(a)$.

Partially ordered by the restriction of $\leq$, the set $P:=\{p\in A:
p=p\sp{2}\}$ of \emph{projections} in $A$ forms an \emph{orthomodular
lattice} (OML) \cite{Beran, Kalm}, \cite[\S 5]{FSynap} with $p\mapsto
p\sp{\perp}:=1-p$ as the \emph{orthocomplementation}. The meet (greatest
lower bound) and join (least upper bound) of projections $p,q\in P$
are denoted by $p\wedge q$ and $p\vee q$, respectively. The projections
$p,q\in P$ are \emph{orthogonal}, in symbols $p\perp q$, iff $p\leq
q\sp{\perp}$, and it turns out that $p\perp q\Rightarrow p+q=p\vee q$.
A minimal nonzero projection in $P$ is called an \emph{atom}. If $p,q
\in P$ and $p$ is an atom, then either $p\wedge q=p$ (i.e., $p\leq q$)
or else $p\wedge q=0$.

Calculations in the OML $P$ are facilitated by the following theorem
\cite[Theorem 5, p. 25]{Kalm}.

\begin{theorem} \label{th:distributive}
For $p,q,r\in P$, if any two of the relations $pCq$, $pCr$, or $qCr$
hold, then $p\wedge(q\vee r)=(p\wedge q)\vee(p\wedge r)$ and $p\vee
(q\wedge r)=(p\vee q)\wedge(p\vee r)$.
\end{theorem}

To each element $a\in A$ is associated a unique projection $a\dg
\in P$ called the \emph{carrier} of $a$ such that, for all $b\in A$,
$ab=0\Leftrightarrow a\dg b=0\Leftrightarrow ba\dg=0\Leftrightarrow
ba=0$.	It turns out that $aa\dg=a\dg a=a$, $a\dg\in CC(a)$, $(a\sp{2})
\dg=a\dg$, $|a|\dg=a\dg$, and if $p\in P$ and $e\in E$, then $p\dg=p$
and $e\leq e\dg$. Also, $0\leq a\leq b\Rightarrow a\dg\leq b\dg$.
Moreover, if $p,q\in P$, then $(pqp)\dg=p\wedge(p\sp{\perp}\vee q)$
\cite[Theorem 5.6]{FSynap}.

We shall have use for the next two lemmas which follow from \cite
[Lemma 4.1]{SymSA} and \cite[Theorem 5.5]{ComSA}.

\begin{lemma} \label{lm:carrierofsum}
If $0\leq a\sb{1}, a\sb{2},..., a\sb{n}\in A$, then
$(\sum\sb{i=1}\sp{n}a\sb{i})\dg=\bigvee\sb{i=1}\sp{n}(a\sb{i})\dg$.
\end{lemma}

\begin{lemma} \label{lm:carrierofprod}
If $a,b,ab\in A$, then $(ab)\dg=a\dg b\dg=b\dg a\dg=a\dg\wedge b\dg$.
\end{lemma}

The set $E:=\{e\in A:0\leq e\leq 1\}$ of \emph{effect elements}
(or for short, simply \emph{effects}) in $A$ forms a convex
effect algebra \cite{FBEA, GPBB}. If $e\in E$, then the \emph
{orthosupplement of $e$} is denoted and defined by $e\sp{\perp}
:=1-e\in E$. Two effects $e$ and $f$ are \emph{disjoint} iff
the only effect $g\in E$ with $g\leq e,f$ is $g=0$. Every projection
is an effect, i.e., $P\subseteq E$; in fact, $P$ is the extreme boundary
of the convex set $E$.

\begin{lemma} \label{lm:infsupinP}
Let $p,q\in P$. Then{\rm: (i)} The infimum $p\wedge q$ of $p$ and
$q$ in $P$ is also the infimum of $p$ and $q$ in $E$. {\rm (ii)}
The supremum $p\vee q$ of $p$ and $q$ in $P$ is also the supremum
of $p$ and $q$ in $E$.
\end{lemma}

\begin{proof}
(i) Of course $p\wedge q\leq p,q$, and it remains to prove that if
$e\in E$ with $e\leq p,q$, then $e\leq p\wedge q$. But, if $e\leq p,q$,
then $e\dg\leq p,q$, whence $e\leq e\dg\leq p\wedge q$.

(ii) Of course $p,q\leq p\vee q$, and it remains to prove that if $e
\in E$ with $p,q\leq e$, then $p\vee q\leq e$. So assume that $p,q
\leq e$, and therefore that $e\sp{\perp}\leq p\sp{\perp},q\sp{\perp}$.
It follows that $(e\sp{\perp})\dg\leq p\sp{\perp},q\sp{\perp}$, whence
$p,q\leq((e\sp{\perp})\dg)\sp{\perp}\in P$.  Consequently, $p\vee q
\leq((e\sp{\perp})\dg)\sp{\perp}$. But $e\sp{\perp}\leq (e\sp{\perp})
\dg$, so $((e\sp{\perp})\dg)\sp{\perp}\leq e\sp{\perp\perp}=e$, and we
have $p\vee q\leq e$.
\end{proof}

In view of Lemma \ref{lm:infsupinP}, no confusion will result if an
existing infimum (respectively, supremum) in $E$ of effects $e,f\in E$
is denoted by $e\wedge f$ (respectively, by $e\vee f$).

By \cite[Theorem 2.6 (v)]{FSynap}, an effect $e\in E$ is a projection iff
$e$ is sharp, i.e., iff $e$ is disjoint from its own orthosupplement
$e\sp{\perp}$ iff $e\wedge e\sp{\perp}=0$. Moreover, the carrier
$e\dg$ of an effect $e\in E$ is the smallest projection that dominates
$e$, so $E$ is a \emph{sharply dominating} effect algebra \cite{SPGSD}.

The next theorem and its corollary provide useful ways to stipulate that
a projection $p$ either dominates or is dominated by an effect $e$.

\begin{theorem} [{\cite[Theorem 2.4]{FSynap}}] \label{th:effleqproj}
Let $p\in P$ and $e\in E$. Then the following conditions are mutually
equivalent{\rm: (i)} $e\leq p$.
{\rm(ii)} $e=ep=pe$. {\rm(iii)} $e=pep$. {\rm(iv)} $e=ep$. {\rm(v)}
$e=pe$.
\end{theorem}

\begin{corollary} \label{co:projleqeff}
If $p\in P$ and $e\in E$, then the following conditions are mutually
equivalent{\rm: (i)} $p\leq e$. {\rm(ii)} $p=ep=pe$. {\rm(iii)} $p=ep$.
{\rm(iv)} $p=pe$.
\end{corollary}

\begin{proof} In Theorem \ref{th:effleqproj}, replace $e$ by $e
\sp{\perp}=1-e$ and $p$ by $p\sp{\perp}=1-p$. Then $p\leq e
\Leftrightarrow e\sp{\perp}\leq p\sp{\perp}$, $1-e=(1-e)(1-p)
\Leftrightarrow p=ep$, and $1-e=(1-p)(1-e)\Leftrightarrow p=pe$.
\end{proof}

As a consequence of Theorem \ref{th:effleqproj} and its corollary,
if a projection $p$ and an effect $e$ are comparable (i.e, $e\leq p$
or $p\leq e$), then $pCe$. One of the reasons that the order structure
of $E$ is so ``wild" is that the same does not hold for two effects.

\begin{lemma}  \label{lm:eCf}
Suppose that $e,f\in E$ and $p\in P$. Then{\rm: (i)} If $eCf$, then
$ef\in E$ and $ef\leq e,f$. {\rm(ii)} If $pCf$, then $pf=fp=pfp=
p\wedge f$, the infimum of $p$ and $f$ in $E$.
\end{lemma}

\begin{proof}
(i) Assume that $e,f\in E$ and $ef=fe$. By \cite[Lemma 1.5]{FSynap},
$0\leq ef$. Likewise, $0\leq e,1-f$ and $eC(1-f)$, so $0\leq e
(1-f)=e-ef$, whence $ef\leq e\leq 1$, so $ef\in E$. By symmetry,
$ef\leq f$.

(ii) Suppose that $pCf$ and let $g\in E$ with $g\leq p,f$. By (i),
$pf\leq p,f$. Also, by Theorem \ref{th:effleqproj}, $g=pgp$, and as
$g\leq f$, we have $g=pgp\leq pfp=p\sp{2}f=pf$, whence $pf=p\wedge f$.
\end{proof}

In part (i) of Lemma \ref{lm:eCf}, we note that although $ef=fe\in E$,
it is not necessarily the infimum of $e$ and $f$ in $E$. In fact, P.J.
Lahti and M.J. M\c{a}czynski \cite[page 1675]{LM} give an example of an
effect operator $e$ on a two-dimensional Hilbert space such that the
infimum of the commuting effects $e$ and $e\sp{\perp}=1-e$ does not exist
in $E$.

\begin{lemma} \label{lm:effectconds}
Suppose that $e\in A$ with $0\leq e$. Then{\rm: (i)} $e\in E\Rightarrow
0\leq e\sp{2}\leq e\leq 1\Rightarrow e\sp{2}\in E$. {\rm(ii)} $e\sp{2}
\leq 1\Leftrightarrow e\in E$. {\rm(iii)} $e\in E\Rightarrow e-e\sp{2}=
ee\sp{\perp}\in E$.
\end{lemma}

\begin{proof}
(i) If $e\in E$, then $e\sp{2}\leq e$ by Lemma \ref{lm:eCf} (i).

(ii) Suppose that $e\sp{2}\leq 1$. Then $0\leq(1-e)\sp{2}+(1-e\sp{2})
=2(1-e)$, so $e\leq 1$, whence $e\in E$. Conversely, if $e\in E$, then
by (i), $e\sp{2}\leq e\leq 1$.

(iii) If $e\in E$, then $0\leq e-e\sp{2}=e(1-e)=ee\sp{\perp}$ by (i)
and $e-e\sp{2}\leq e\leq 1$, so $e-e\sp{2}\in E$.
\end{proof}

Each element $a\in A$ determines and is determined by a one-parameter
family of projections $(p\sb{a,\lambda})\sb{\lambda\in\reals}$ called
its \emph{spectral resolution} and defined by $p\sb{a,\lambda}:=1-
((a-\lambda1)\sp{+})\dg$ for all $\lambda\in\reals$ \cite[Definition
8.2]{FSynap}. See \cite[\S 8]{FSynap}, especially \cite[Theorem 8.4]
{FSynap} for the basic properties of the spectral resolution. We note
that by \cite[Theorem 8.10]{FSynap}, if $a,b\in A$, then $bCa$ iff
$bCp\sb{a,\lambda}$ for all $\lambda\in\reals$.

By \cite[Theorem 8.4 (vii)]{FSynap}, the spectral resolution $(p\sb{a,
\lambda})\sb{\lambda\in\reals}$ is uniquely determined by the corresponding
\emph{rational spectral resolution} $(p\sb{a,\mu})\sb{\mu\in\rationals}$
according to the formula
\[
p\sb{a,\lambda}=\bigwedge\{p\sb{a,\mu}:\lambda\leq\mu\in\rationals\}
 \text{\ \ for each\ }\lambda\in\reals.
\]

\begin{remark} \label{rk:RatSpecCommute}
If $a\in A$ and $q\in P$, then since commutativity of projections is
preserved under the formation of arbitrary existing infima, the formula
above implies that $qCa$ iff $qCp\sb{a,\mu}$ for all $\mu\in\rationals$.
\end{remark}

Let $q\in P$. Then with the partial order and operations inherited from
$A$, the subset
\[
qAq:=\{qaq:a\in A\}=\{a\in A:a=qaq\}=\{a\in A:a=qa=aq\}\subseteq A
\]
is a synaptic algebra in its own right with unity element $q$ and
with $qRq$  as its enveloping algebra \cite[Theorem 4.10]{FSynap}. The
OML of projections in $qAq$ is $P[0,q]:=\{v\in P:v\leq q\}$ with the
orthocomplementation $v\mapsto v\sp{\perp\sb{q}}:=v\sp{\perp}\wedge q$.
Likewise, the set of all effects in $qAq$ is $E[0,q]:=\{f\in E:f
\leq q\}$ with the orthosupplementation $f\mapsto f\sp{\perp\sb{q}}:=
q-f=(1-f)q=f\sp{\perp}q=qf\sp{\perp}=f\sp{\perp}\wedge q$ (Lemma \ref
{lm:eCf} (ii)). Let $a\in qAq$. Then $|a|$, $a\sp{+}$, $a\dg$, and
if $0\leq a$, $a\sp{1/2}$, belong to $qAq$ and coincide with the
absolute value, the positive part, the carrier, and the square root
of $a$, respectively, as calculated in $qAq$.

\begin{lemma} \label{lm:SRofqaq}
Let $a\in A$, $f\in E$, and $q\in P$. Then{\rm: (i)} If $qCa$, then the
spectral resolution of $qa=aq\in qAq$ as calculated in $qAq$ is given by
$(qp\sb{a,\lambda})\sb{\lambda\in\reals}=(p\sb{a,\lambda}\wedge q)\sb
{\lambda\in\reals}$. {\rm(ii)} If $qCf$, then the spectral resolution
of $qf=fq=f\wedge q\in qAq$, as calculated in $qAq$, is given by $(qp\sb{f,
\lambda})\sb{\lambda\in\reals}=(p\sb{f,\lambda}\wedge q)\sb{\lambda\in
\reals}$.
\end{lemma}

\begin{proof}
Part (i) is proved by a direct calculation using \cite[Definition 8.2
and Theorem 4.10]{FSynap} and the fact that $qCa$ implies $qCp\sb{a,
\lambda}$, whence $p\sb{a,\lambda}\wedge q=qp\sb{a,\lambda}$ for all
$\lambda\in\reals$. Part (ii) follows from (i) and Lemma \ref{lm:eCf}
(ii).
\end{proof}

\begin{lemma} \label{lm:pAp,patom}
Suppose that $p$ is an atom in $P$. Then{\rm: (i)} $pAp=\{\lambda p:\lambda
\in\reals\}$. {\rm(ii)} If $a\in A$, there exists a unique $\lambda\in\reals$
such that $pap=\lambda p$. {\rm(iii)} If $f\in E$ and $pfp=\lambda p$, then
$0\leq\lambda\leq 1$.
\end{lemma}

\begin{proof}
(i) Since $p$ is an atom, it follows that $0$ and $p\not=0$ are the only projections
in the synaptic algebra $pAp$, from which, using spectral theory in $pAp$, (i)
follows. Part (ii) follows from the fact that $p\not=0$, and (iii) is a
consequence of $0\leq f\leq 1\Rightarrow 0\leq pfp\leq p1p=p\sp{2}=p\leq 1$.
\end{proof}

An element $u\in A$ is said to be a \emph{symmetry} \cite{SymSA} iff
$u\sp{2}=1$, and a \emph{partial symmetry} is an element $t\in A$
such that $t\sp{2}\in P$. As a consequence of the uniqueness theorem
for square roots, a projection is the same thing as a partial symmetry
$p$ such that $0\leq p$. If $t\in A$ is a partial symmetry, then
$u:=t+(t\sp{2})\sp{\perp}$ is a symmetry called the \emph{canonical
extension} of $t$.

If $a\in A$ there is a uniquely determined partial symmetry $t\in A$,
called the \emph{signum} of $a$, such that $t\sp{2}=a\dg$ and $a=|a|t$.
Moreover, $t\in CC(a)$, $t\dg=a\dg$, and if $u=t+(t\sp{2})\sp{\perp}$ is
the canonical extension of $t$ to a symmetry, then $u\in CC(a)$ and
$a=|a|u=u|a|$. The latter formula is called the \emph{polar
decomposition} of $a$. It turns out that the symmetry $u$ in the polar
decomposition of $a$ is uniquely determined.

If $a,b\in A$ and $u\in A$ is a symmetry, it is not difficult to verify
that $a\leq b\Leftrightarrow uau\leq ubu$ and that $ua\dg u=(uau)\dg$.

Two projections $p,q\in P$ are \emph{exchanged by a symmetry} $u\in A$
iff $upu=q$ (whence, automatically, $uqu=p$) and they are \emph{exchanged
by a partial symmetry} $t\in A$ iff $tpt=q$ and $tqt=p$. If $p$ and $q$
are exchanged by a partial symmetry $t$, then they are exchanged by the
canonical extension $u:=t+(t\sp{2})\sp{\perp}$ of $t$ to a symmetry.

If $p\in P$ and $a\in A$, then by direct calculation using the
fact that $p\sp{\perp}=1-p$, one obtains the well-known \emph{Peirce
decomposition} of $a$ with respect to $p$, namely
\[
a=pap+pap\sp{\perp}+p\sp{\perp}ap+p\sp{\perp}ap\sp{\perp}.
\]
We refer to $pap+p\sp{\perp}ap\sp{\perp}$ as the \emph{diagonal part}
of $a$ with respect to $p$ and to $pap\sp{\perp}+p\sp{\perp}ap$ as
the \emph{off-diagonal part} of $a$ with respect to $p$. We note
that $pap$, $p\sp{\perp}ap\sp{\perp}$, and the diagonal part $pap+
p\sp{\perp}ap\sp{\perp}$ of $a$ belong to $A$. Also, although
$pap\sp{\perp}$ and $p\sp{\perp}ap$ belong to the enveloping algebra
$R$, but not necessarily to $A$, the off-diagonal part $pap\sp
{\perp}+p\sp{\perp}ap$ belongs to $A$.

\begin{lemma}[{\cite[Theorem 2.12]{FJP2proj}}] \label{lm:diagzero}
If $0\leq a\in A$ and $p\in P$, then $a=0$ iff the diagonal part of
$a$ with respect to $p$ is zero.
\end{lemma}

\begin{lemma}  \label{lm:offdiagzero}
Let $a\in A$ and $p\in P$. Then the following conditions are mutually
equivalent{\rm: (i)} $pCa$. {\rm(ii)} The off-diagonal part of $a$
with respect to $p$ is zero. {\rm(iii)} $pa\in A$. {\rm(iv)} $ap\in A$.
{\rm(v)} $pap\sp{\perp}=0$. {\rm(vi)} $p\sp{\perp}ap=0$.
\end{lemma}

\begin{proof}
The equivalence (i) $\Leftrightarrow$ (ii) follows from {\cite
[Theorem 2.12]{FJP2proj}}. If $pa\in A$, then since $pa+ap\in A$,
we have $ap=(pa+ap)-pa\in A$; similarly, $ap\in A\Rightarrow
pa\in A$, and we have (iii) $\Leftrightarrow$ (iv).  To prove that
(i) $\Leftrightarrow$ (iii), note that $pCa\Rightarrow pa=ap\in A$.
Conversely, suppose that $pa\in A$. Then, since (iii) $\Leftrightarrow$
(iv), $ap\in A$. Also, $(1-p)pa=0$, so $pa(1-p)=0$, and we have $pa=pap$.
Similarly, $ap(1-p)=0$, so $(1-p)ap=0$, i.e., $ap=pap$, whence $pa=pap=ap$.
This proves that (i) $\Leftrightarrow$ (iii), and it follows that
conditions (i)--(iv) are mutually equivalent.

If (i) holds, then $pap\sp{\perp}=app\sp{\perp}=0$, so (i) $\Rightarrow$
(v). Conversely, if (v) holds, then $0=pap\sp{\perp}=pa(1-p)=pa-pap$,
so $pa=pap\in A$, and we have (v) $\Rightarrow$ (iii). Similarly,
(i) $\Rightarrow$ (vi) $\Rightarrow$ (iv).
\end{proof}

\section{A projection and an effect} \label{sc:p&e}

\begin{assumption}
For the remainder of this article we assume that $p\in P$, and
$e\in E$.
\end{assumption}
In this section we associate with the pair $p,e$ four special
effects, $c$, $s$, $j$, and $b$ (Definitions \ref{df:cs}, \ref{df:j},
and \ref{df:b}) and a symmetry $k$ (Definition \ref{df:symmetryk}). Using
$c$, $s$, $j$, $b$, and $k$, we rewrite the Peirce decomposition of $e$
with respect to $p$, thus obtaining the \emph{CBS-decomposition of
$e$ with respect to $p$} (Theorem \ref{th:CBSdecomp}).

In the next definition we generalize to the present case the definitions
of the cosine and sine effects for a projection $q$ with respect to
the projection $p$ \cite[Definition 4.2]{FJP2proj}.

\begin{definition} \label{df:cs}
Since $0\leq e, e\sp{\perp}$, we have $0\leq pep+p\sp{\perp}e\sp{\perp}
p\sp{\perp}$ and $0\leq pe\sp{\perp}p+p\sp{\perp}ep\sp{\perp}$. Thus, we
define the \emph{cosine} effect $c$ and the \emph{sine} effect $s$ for
$e$ with respect to the  projection $p$ as follows:
\[
\text{(1)\ } c:=(pep+p\sp{\perp}e\sp{\perp}p\sp{\perp})\sp{1/2}.
 \text{\ \ \ (2)\ } s:=(pe\sp{\perp}p+p\sp{\perp}ep\sp{\perp})\sp{1/2}.
\]
\end{definition}

\begin{lemma} \label{lm:ecsProps}
{\rm(i)} $c\sp{2}=1-p+pe+ep-e$. {\rm(ii)} $s\sp{2}=p-pe-ep+e$. {\rm(iii)}
$c\sp{2}+s\sp{2}=1$. {\rm(iv)} $c\sp{2}p=pc\sp{2}=pep$ and $s\sp{2}p
\sp{\perp}=p\sp{\perp}s\sp{2}= p\sp{\perp}ep\sp{\perp}$. {\rm(v)} $c,s
\in C(p)$ and $cCs$. {\rm(vi)} $c,s,cs,c\sp{2},s\sp{2},c\sp{2}s\sp{2}
\in E$, $c\sp{2}\leq c$, and $s\sp{2}\leq s$.
\end{lemma}

\begin{proof}
Parts (i) and (ii) follow from straightforward calculations using the facts that
$p\sp{\perp}=1-p$ and $e\sp{\perp}=1-e$. Obviously, (iii) follows from (i) and
(ii).

By (i) we have $c\sp{2}p=p-p+pep+ep-ep=pep$ and $pc\sp{2}=p-p+pe+pep-pe=pep$.
Using (ii), a similar calculation yields $s\sp{2}p\sp{\perp}=p\sp{\perp}s\sp{2}=
p\sp{\perp}ep\sp{\perp}$, and we have (iv).

As $0\leq c,s$, it follows that $C(c)=C(c\sp{2})$ and $C(s)=C(s\sp{2})$. By (iv),
$p\in C(c\sp{2})$ and $p\in C(s\sp{2})$, whence $pCc$ and $pCs$, and (v) is proved.

We have $0\leq c,s$ and since $c\sp{2}, s\sp{2}\leq c\sp{2}+s\sp{2}=1$, we have
$c\sp{2}, s\sp{2}\leq 1$, whence by Lemma \ref{lm:effectconds}, $c\sp{2}\leq c
\in E$, and $s\sp{2}\leq s\in E$. Thus, since $c,s\in E$ and $cCs$, Lemma
\ref{lm:eCf} (i) implies that $cs\in E$, and (vi) is proved.
\end{proof}

As $e\in E$, we have $e\sp{2}\in E$ with $e-e\sp{2}=ee\sp{\perp}\in E$
(Lemma \ref{lm:effectconds} (iii)), whence $p(e-e\sp{2})p+p\sp{\perp}
(e-e\sp{2})p\sp{\perp}\geq 0$.

\begin{definition} \label{df:j}
We define $j\in A$ by
\[
j:=(p(e-e\sp{2})p+p\sp{\perp}(e-e\sp{2})p\sp{\perp})\sp{1/2},
\]
i.e., $0\leq j$ and $j\sp{2}$ is the diagonal part of $e-e\sp{2}=ee\sp
{\perp}$ with respect to $p$.
\end{definition}

In the next lemma we obtain an important relation between $c\sp{2}
s\sp{2}$, the diagonal part $j\sp{2}$ of $e-e\sp{2}$ with respect to $p$,
and the square of the off-diagonal part $pep\sp{\perp}+p\sp{\perp}ep$
of $e$ with respect to $p$.

\begin{lemma} \label{lm:squareofcs}
$c\sp{2}s\sp{2}=(cs)\sp{2}=j\sp{2}+(pep\sp{\perp}+p\sp{\perp}ep)\sp{2}$.
\end{lemma}

\begin{proof} \setcounter{equation}{0}
By parts (i) and (ii) of Lemma \ref{lm:ecsProps},
\[
c\sp{2}s\sp{2}=(1-p+pe+ep-e)(p-pe-ep+e)=(p-pe-ep+e)-(p-pe-ep+e)\sp{2}
\]
\[
=p-pe-ep+e-p+pe+pep-pe+pep-pepe-pe\sp{2}p+pe\sp{2}
\]
\[
+ep-epe-epep+epe-ep+epe+e\sp{2}p-e\sp{2}
\]
\begin{equation} \label{eq:squareofcs01}
=e-e\sp{2}+2p(e-e\sp{2})p-(e-e\sp{2})p-p(e-e\sp{2})+pe\sp{2}p+epe-epep-pepe.
\end{equation}
Also,
\begin{equation} \label{eq:squareofcs02}
(pep\sp{\perp}+p\sp{\perp}ep)\sp{2}=pep\sp{\perp}ep+p\sp{\perp}
 epep\sp{\perp}=pe\sp{2}p+epe-epep-pepe
\end{equation}
and
\[
j\sp{2}=p(e-e\sp{2})p+(1-p)(e-e\sp{2})(1-p)
\]
\begin{equation} \label{eq:squareofcs03}
=e-e\sp{2}+2p(e-e\sp{2})p-(e-e\sp{2})p-p(e-e\sp{2}).
\end{equation}
Combining Equations (\ref{eq:squareofcs01}), (\ref{eq:squareofcs02}),
and (\ref{eq:squareofcs03}), we obtain the desired result.
\end{proof}

\begin{definition} \label{df:b}
By Lemma \ref{lm:squareofcs}, $0\leq c\sp{2}s\sp{2}-j\sp{2}$, which
enables us to define
\[
b:=(c\sp{2}s\sp{2}-j\sp{2})\sp{1/2}.
\]
We refer to $b$ as the \emph{commutator effect} for the pair $p,e$
(see Lemma \ref{lm:pCe} below).
\end{definition}

\begin{theorem} \label{th:bProps}
{\rm(i)\ } $pCj$\text{\ and\ }$pCb$. {\ \rm(ii)\ }$b\in E$. {\ \rm(iii)\ }
 $b=|pep\sp{\perp}+p\sp{\perp}ep|$.
\end{theorem}

\begin{proof}
(i) Since
\[
p(p(e-e\sp{2})p+p\sp{\perp}(e-e\sp{2})p\sp{\perp})=p(e-e\sp{2})p
=(p(e-e\sp{2})p+p\sp{\perp}(e-e\sp{2})p\sp{\perp})p,
\]
we have $pC(p(e-e\sp{2})p+p\sp{\perp}(e-e\sp{2})p\sp{\perp})$,
and since
\[
j=(p(e-e\sp{2})p+p\sp{\perp}(e-e\sp{2})p\sp{\perp})\sp{1/2}
\]
it follows that $pCj$. Also, by Lemma \ref{lm:ecsProps}
(v), $pC(c\sp{2}s\sp{2})$, and therefore\linebreak $pC(c\sp{2}s
\sp{2}-j\sp{2})$. As $b=(c\sp{2}s\sp{2}-j\sp{2})\sp{1/2}$, it
follows that $pCb$.

(ii) Evidently, $0\leq b$. Also by Lemma \ref{lm:ecsProps} (vi),
$b\sp{2}\leq c\sp{2}s\sp{2}\leq 1$, and it follows from Lemma
\ref{lm:effectconds} (ii) that $b\in E$.

Part (iii) follows immediately from Lemma \ref{lm:squareofcs} and
Definition \ref{df:b}.
\end{proof}

\begin{definition} \label{df:symmetryk}
As per Theorem \ref{th:bProps} (iii), we define the symmetry $k$
by polar decomposition of $pep\sp{\perp}+p\sp{\perp}ep$, so that
\[
pep\sp{\perp}+p\sp{\perp}ep=|pep\sp{\perp}+p\sp{\perp}ep|k=bk=kb
\]
where $k\in CC(pep\sp{\perp}+p\sp{\perp}ep)$.
\end{definition}

\begin{theorem}[CBS-decomposition] \label{th:CBSdecomp}
\[
e=c\sp{2}p+bk+s\sp{2}p\sp{\perp},
 \text{\ where}
\]
\begin{enumerate}
\item $pep=c\sp{2}p=pc\sp{2}$ and $p\sp{\perp}ep\sp{\perp}=s\sp{2}p
\sp{\perp}=p\sp{\perp}s\sp{2}$.
\item $b=|pep\sp{\perp}+p\sp{\perp}ep|=(c\sp{2}s\sp{2}-j\sp{2})
 \sp{1/2}\in E$.
\item $k$ is a symmetry and $pep\sp{\perp}+p\sp{\perp}ep=bk=kb$.
\item $cCp$, $sCp$, $cCs$, $bCp$, and $k\in CC(pep\sp{\perp}+
 p\sp{\perp}ep)$.
\item $pbk=bpk=bkp\sp{\perp}=pep\sp{\perp}$, whence $b(pk-kp\sp{\perp})
 =b\dg(pk-kp\sp{\perp})=0$.
\item $p\sp{\perp}bk=bp\sp{\perp}k=bkp=p\sp{\perp}ep$.
\end{enumerate}
\end{theorem}

\begin{proof}
Parts (i), (ii), (iii), and the formula $e=c\sp{2}p+bk+s\sp{2}
p\sp{\perp}$ follow from Lemma \ref{lm:ecsProps} (iv), Lemma
\ref{th:bProps} (iii), Definition \ref{df:symmetryk}, and the Pierce
decomposition of $e$ with respect to $p$. Part (iv) is a consequence
of Lemma  \ref{lm:ecsProps} (v), Lemma \ref{th:bProps} (i), and
Definition \ref{df:symmetryk}.

By (iii) and the fact that $bCp$, we have $bpk=pbk=p(pep\sp{\perp}+
p\sp{\perp}ep)=pep\sp{\perp}=(pep\sp{\perp}+p\sp{\perp}ep)p\sp{\perp}
=bkp\sp{\perp}$, whence $b(pk-kp\sp{\perp})=pep\sp{\perp}-pep\sp{\perp}
=0$, proving (v). Part (vi) follows immediately from (v).
\end{proof}

As a consequence of the next lemma, in case $e$ is a projection,
then the CBS-decomposition theorem reduces to the generalized
CS-decomposition theorem (\cite[Theorem 5.6]{FJP2proj}).

\begin{lemma} \label{lm:einP}
The following conditions are mutually equivalent{\rm: (i)}
$e$ is a projection. {\rm(ii)} $j=0$. {\rm(iii)} $b=cs$.
\end{lemma}

\begin{proof}
By Lemma \ref{lm:diagzero} (i), $e-e\sp{2}=0$ iff $j=0$,
whence (i) $\Leftrightarrow$ (ii). That (ii) $\Leftrightarrow$
(iii) is an immediate consequence of Definition \ref{df:b}.
\end{proof}

\begin{lemma} \label{lm:pCe}
The following conditions are mutually equivalent{\rm: (i)}
$pCe$. {\rm(ii)} $b=0$. {\rm(iii)} $b\dg=0$. {\rm(iv)} $cs=j$.
{\rm(v)} $e=c\sp{2}p+s\sp{2}p\sp{\perp}$.
\end{lemma}

\begin{proof}
The equivalence (i) $\Leftrightarrow$ (ii) follows from Lemma
\ref{lm:offdiagzero} and Theorem \ref{th:bProps} (iii),
and the equivalence (ii) $\Leftrightarrow$ (iii) is obvious.
The equivalence (ii) $\Leftrightarrow$ (iv) is a consequence
of Definition \ref{df:b}, so (i)--(iv) are mutually equivalent.
That (ii) $\Rightarrow$ (v) follows from Theorem \ref{th:CBSdecomp},
and since $p$ commutes with both $c\sp{2}$ and $s\sp{2}$, it is
clear that (v) $\Rightarrow$ (i).
\end{proof}

\begin{definition} \label{df:componentofa}
If $a\in A$, $q\in P$, and $aCq$, then the \emph{component of $a$
in the synaptic algebra $qAq$} is denoted and defined by $a\sb{q}
:=aq=qa=qaq\in qAq$.
\end{definition}

If $a\in A$, $q\in P$, and $aCq$, it is easy to see that $a=a\sb{q}
+a\sb{q\sp{\perp}}$ is the unique decomposition of $a$ as a sum of
an element in $qAq$ and an element in $q\sp{\perp}Aq\sp{\perp}$. This
decomposition can be useful in deducing properties of $a$ from
properties of its components $a\sb{q}\in qAq$ and $a\sb{q\sp{\perp}}
\in q\sp{\perp}Aq\sp{\perp}$.

\begin{lemma}  \label{lm:components}
Let $f\in E$, $q\in P$, and suppose that $fCq$. Then{\rm: (i)}
The component $f\sb{q}=fq=qf=f\wedge q$ is an effect in $qAq$.
{\rm(ii)} The orthosupplement of $f\sb{q}$ in $E[0,q]$ is the
component of $f\sp{\perp}$ in $qAq$, i.e., $f\sb{q}\sp{\,
\perp\sb{q}}=qf\sp{\perp}=f\sp{\perp}q=f\sp{\perp}\wedge q=
(f\sp{\perp})\sb{q}$.
\end{lemma}

\begin{proof}
By Lemma \ref{lm:eCf} (ii), $qf=fq=f\wedge q\in E[0,q]$, proving
(i). Also, $f\sb{q}\sp{\,\perp\sb{q}}=(fq)\sp{\perp}q=(1-fq)q
=q-fq=(1-f)q=f\sp{\perp}q=f\sp{\perp}\wedge q=(f\sp{\perp})\sb{q}$,
proving (ii).
\end{proof}

\begin{theorem} \label{th:esubq}
For $p\in P$ and $e\in E$, suppose that $q\in P$ with $qCp$ and $qCe$.
Then{\rm: (i)} $q$ commutes with $c,s,b,$ and $k$. {\rm(ii)} The cosine,
sine, and commutator effects for $e\sb{q}$ with respect to $p\sb{q}=
pq=qp=p\wedge q$ as calculated in $qAq$ are $c\sb{q}=cq=qc=c\wedge q$, $s
\sb{q}=sq=qs=s\wedge q$, and $b\sb{q}=bq=qb=b\wedge q$, respectively.
{\rm(iii)} The CBS-decomposition of $e\sb{q}$ with respect to $p\sb{q}$
in $qAq$ is $e\sb{q}=c\sb{q}\sp{\,2}p\sb{q}+b\sb{q}k\sb{q}+s\sb{q}
\sp{\,2}p\sb{q}\sp{\,\perp\sb{q}}=q(c\sp{2}p+bk+s\sp{2}p\sp{\perp})
=(c\sp{2}p+bk+s\sp{2}p\sp{\perp})q$.
\end{theorem}

\begin{proof}
(i) As $c=(pep+p\sp{\perp}e\sp{\perp}p\sp{\perp})\sp{1/2}\in CC(pep+
p\sp{\perp}e\sp{\perp}p\sp{\perp})$, we have $qCc$ and similarly
$qCs$. Likewise, $qCb$ follows from $b=|pep\sp{\perp}+p\sp{\perp}ep|$
(Theorem \ref{th:bProps} (iii)), and $qCk$ follows from $k\in CC(pep
\sp{\perp}+p\sp{\perp}ep)$.

(ii) Obviously, $p\sb{q}e\sb{q}p\sb{q}=qpep=pepq$. Also, as $pCq$, we
have $p\sb{q}\sp{\,\perp\sb{q}}=p\sp{\perp}\wedge q=p\sp{\perp}q=
qp\sp{\perp}$. Moreover, $e\sb{q}\sp{\,\perp\sb{q}}=qe\sp{\perp}=
e\sp{\perp}q=qe\sp{\perp}q$. Therefore the cosine effect for
$e\sb{q}$ with respect to $p\sb{q}$ in $qAq$ is
\[
(p\sb{q}e\sb{q}p\sb{q}+p\sb{q}\sp{\,\perp\sb{q}}e\sb{q}\sp{\,\perp
\sb{q}}p\sb{q}\sp{\,\perp\sb{q}})\sp{1/2}=(pepq+p\sp{\perp}e
\sp{\perp}p\sp{\perp}q)\sp{1/2}=cq=c\sb{q}.
\]
Similar computations take care of $s\sb{q}$ and $b\sb{q}$. Part
(iii) follows from (ii).
\end{proof}

\section{Carriers and projection-free effects}

The assumptions and notation of Section \ref{sc:p&e} remain in force.
In this section we derive some information about the carriers of the
effects $e$, $c$, $s$, $j$, and $b$. Also, we introduce two special
projections, $z$ and $t$, associated with the effect $e$ (Definition
\ref{df:z} below).

If $f\in E$, $q\in P$, and $q\leq f$, we say that $q$ is a
\emph{subprojection} of $f$; likewise, if $g\in E$ and $g\leq f$,
we say that $g$ is a \emph{subeffect} of $f$.

\begin{definition} \label{df:ProjFree}
If $f\in E$ and the only subprojection of $f$ is $0$, we say that
$f$ is \emph{projection free}.
\end{definition}

\noindent Obviously, every subeffect of a projection-free effect
is projection free.

\begin{lemma} \label{lm:largestsubpro}
{\rm(i)} If $f\in E$, then $((f\sp{\perp})\dg)\sp{\perp}$ is the
largest subprojection of $f$. {\rm(ii)} $f$ is projection free
iff $(f\sp{\perp})\dg=1$. {\rm(iii)} $f\sp{\perp}$ is projection
free iff $f\dg=1$. {\rm(iv)} $f-((f\sp{\perp})\dg)\sp{\perp}$
and $f\sp{\perp}-(f\dg)\sp{\perp}$ are projection-free effects.
\end{lemma}

\begin{proof}
Part (i) follows from the fact that $(f\sp{\perp})\dg$ is the
smallest projection that dominates $f\sp{\perp}$ \cite
[Theorem 2.10 (iv)]{FSynap}, and parts (ii) and (iii) are
immediate consequences of (i).

(iv) By (i), $((f\sp{\perp})\dg)\sp{\perp}$ is a subprojection of
$f$, so $g:=f-((f\sp{\perp})\dg)\sp{\perp}$ is an effect. We have
$g\sp{\perp}=1-f+((f\sp{\perp})\dg)\sp{\perp}=f\sp{\perp}+((f\sp
{\perp})\dg)\sp{\perp}$, whence by Lemma \ref{lm:carrierofsum},
$(g\sp{\perp})\dg=(f\sp{\perp})\dg\vee((f\sp{\perp})\dg)\sp{\perp}
=1$, so $g$ is projection free by (ii). Similarly, $f\sp{\perp}
-(f\dg)\sp{\perp}$ is a projection-free effect.
\end{proof}

\begin{definition} \label{df:z}
In what follows, $z:=((e\sp{\perp})\dg)\sp{\perp}$ is the largest
subprojection of $e$ and $t:=((e\sp{\perp\,\perp})\dg)\sp{\perp}=
(e\dg)\sp{\perp}$ is the largest subprojection of $e\sp{\perp}$.
\end{definition}

\noindent We note that $(e\sp{\perp})\dg=z\sp{\perp}$ and
$e\dg=t\sp{\perp}$. Evidently, $e\in P\Leftrightarrow e=z=t
\sp{\perp}$.

\begin{theorem} \label{th:largestsubproj}
\
\begin{enumerate}
\item $z,t\in P\cap CC(e)$, $z\leq e\leq e\dg$, $t\leq e\sp{\perp}\leq
 (e\sp{\perp})\dg$, and $e-z,\, e\sp{\perp}-t\in E$.
\item $e$ is projection free iff $z=0$ iff $(e\sp{\perp})\dg=1$ and
 $e\sp{\perp}$ is projection free iff $t=0$ iff $e\dg=1$.
\item $z\perp t$, i.e., $(e\dg)\sp{\perp}\leq(e\sp{\perp})\dg$.
\item $e-z$ and $e\sp{\perp}-t$ are projection-free effects.
\item $(e-z)\dg=e\dg-z=e\dg\wedge z\sp{\perp}=t\sp{\perp}\wedge z
 \sp{\perp}=(t\vee z)\sp{\perp}=(t+z)\sp{\perp}$.
\item $(e\sp{\perp}-t)\dg=(e-z)\dg=(t+z)\sp{\perp}$.
\end{enumerate}
\end{theorem}

\begin{proof}
(i) By \cite[Theorem 2.10 (vi)]{FSynap}, $z\sp{\perp}=(e\sp{\perp})
\dg\in CC(e\sp{\perp})$, from which $z\in P\cap CC(e)$ follows;
similarly, $t\in P\cap CC(e)$.

(ii) Part (ii) follows immediately from Lemma \ref{lm:largestsubpro}
(ii).

(iii) Since $e\leq e\dg$, it follows that $(e\dg)\sp{\perp}\leq e
\sp{\perp}$, and therefore $t=(e\dg)\sp{\perp}=((e\dg)\sp{\perp})\dg
\leq (e\sp{\perp})\dg=z\sp{\perp}$.

(iv) Part (iv) follows immediately from Lemma \ref{lm:largestsubpro}
(iv).

(v) We have $e=z+(e-z)$, where $z, e-z\in E$, whence by Lemma
\ref{lm:carrierofsum}, $e\dg=z\dg\vee(e-z)\dg=z\vee(e-z)\dg$.
Also, $e-z\leq 1-z=z\sp{\perp}$, whence $(e-z)\dg\leq z\sp{\perp}$,
and it follows that $e\dg=z\vee(e-z)\dg=z+(e-z)\dg$, so $(e-z)
\dg=e\dg-z$. Also, since $z\leq e\dg$, we have $e\dg-z=e\dg\wedge
z\sp{\perp}=t\sp{\perp}\wedge z\sp{\perp}$, and the remaining
equalities follow from De\,Morgan and the fact that $z\perp t$.

(vi) Proceeding as in the proof of (v), we have $(e\sp{\perp}-t)
\dg=(e\sp{\perp})\dg-t=(e\sp{\perp})\dg\wedge t\sp{\perp}=z\sp{\perp}
\wedge t\sp{\perp}=t\sp{\perp}\wedge z\sp{\perp}=(e-z)\dg=(t+z)
\sp{\perp}$.
\end{proof}

\begin{corollary} \label{co:zPropscor}
{\rm(i)} $e-e\sp{2}=(e-z)-(e-z)\sp{2}\leq e-z$. {\rm(ii)} $e-e\sp{2}$
is projection free. {\rm(iii)} $(e-e\sp{2})\dg=t\sp{\perp}\wedge z
\sp{\perp}=(e-z)\dg=e\dg-z$.
\end{corollary}

\begin{proof}
(i) Since $z\leq e$ and $z\in P$, we have $ze=ez=z$, whence
$(e-z)-(e-z)\sp{2}=e-z-(e\sp{2}-ez-ze+z)=e-e\sp{2}$ and
$e-e\sp{2}\leq e-z$.

(ii) By Theorem \ref{th:largestsubproj} (iv), $e-z$ is
projection free; by part (i), $e-e\sp{2}$ is a subeffect of
$e-z$; therefore $e-e\sp{2}$ is projection free.

(iii) By Lemma \ref{lm:carrierofprod}, $(e-e\sp{2})\dg=
(ee\sp{\perp})\dg=e\dg(e\sp{\perp})\dg=t\sp{\perp}z\sp{\perp}
=t\sp{\perp}\wedge z\sp{\perp}$.
\end{proof}

\begin{theorem} \label{th:ecarcs}
\
\begin{enumerate}
\item $c\dg=(p\vee z\sp{\perp})\wedge(p\sp{\perp}\vee t
 \sp{\perp})$ and $s\dg=(p\vee t\sp{\perp})\wedge(p\sp{\perp}
 \vee z\sp{\perp})$.
\item $(cs)\dg=c\dg s\dg=s\dg c\dg=c\dg\wedge s\dg$.
\item $(c\sp{2}s\sp{2})\dg=(cs)\dg=(p\vee z\sp{\perp})\wedge
 (p\vee t\sp{\perp})\wedge(p\sp{\perp}\vee z\sp{\perp})\wedge
 (p\sp{\perp}\vee t\sp{\perp})$.
\item $j\dg=(p\vee(t\sp{\perp}\wedge z\sp{\perp}))\wedge(p\sp{\perp}
 \vee(t\sp{\perp}\wedge z\sp{\perp}))$.
\item ${s\dg}\sp{\perp}\leq c\sp{2}\leq c$ and ${c\dg}\sp{\perp}\leq s
 \sp{2}\leq s$.
\item $(s\dg)\sp{\perp}e=e(s\dg)\sp{\perp}=(s\dg)\sp{\perp}\wedge e
 =(s\dg)\sp{\perp}p=p(s\dg)\sp{\perp}=(s\dg)\sp{\perp}\wedge p$.
\item $(c\dg)\sp{\perp}e=e(c\dg)\sp{\perp}=(c\dg)\sp{\perp}\wedge e
 =(c\dg)\sp{\perp}p\sp{\perp}=p\sp{\perp}(c\dg)\sp{\perp}=(c\dg)
 \sp{\perp}\wedge p\sp{\perp}$.
\end{enumerate}
\end{theorem}

\begin{proof}
\setcounter{equation}{0}
(i) Since $0\leq pqp, p\sp{\perp}q\sp{\perp}p\sp{\perp}$, we infer
from \cite[Theorem 4.9 (v)]{FSynap} and Lemma \ref{lm:carrierofsum}
that
\[
c\dg=[(pep+p\sp{\perp}e\sp{\perp}p\sp{\perp})\sp{1/2}]\dg=(pep+
 p\sp{\perp}e\sp{\perp}p\sp{\perp})\dg=(pep)\dg\vee(p\sp{\perp}e
 \sp{\perp}p\sp{\perp})\dg
\]
\[
=(pe\dg p)\dg\vee(p\sp{\perp}(e\sp{\perp})\dg p\sp{\perp})\dg
 =(pt\sp{\perp}p)\dg\vee(p\sp{\perp}z\sp{\perp}p\sp{\perp})\dg
\]
\begin{equation} \label{eq:ecar01}
=[p\wedge(p\sp{\perp}\vee t\sp{\perp})]\vee[p\sp{\perp}\wedge(p\vee z
 \sp{\perp})]=[p\wedge(p\sp{\perp}\vee t\sp{\perp})]\vee w,
\end{equation}
where $w:=p\sp{\perp}\wedge(p\vee z\sp{\perp})$. Now $pC(p\sp{\perp}
\vee t\sp{\perp})$ and $pCw$, whence
\begin{equation} \label{eq:ecar02}
[p\wedge(p\sp{\perp}\vee t\sp{\perp})]\vee w=(p\vee w)\wedge
 (p\sp{\perp}\vee t\sp{\perp}\vee w).
\end{equation}
But $pCp\sp{\perp}$ and $pC(p\vee z\sp{\perp})$, whence
\begin{equation} \label{eq:ecar03}
p\vee w=p\vee[p\sp{\perp}\wedge(p\vee z\sp{\perp})]=(p\vee p
 \sp{\perp})\wedge(p\vee p\vee z\sp{\perp})=p\vee z\sp{\perp}.
\end{equation}
Furthermore, since $w\leq p\sp{\perp}$,
\begin{equation} \label{eq:ecar04}
p\sp{\perp}\vee t\sp{\perp}\vee w=p\sp{\perp}\vee t\sp{\perp}.
\end{equation}
By Equations (\ref{eq:ecar03}) and (\ref{eq:ecar04}),
\[
(p\vee w)\wedge(p\sp{\perp}\vee t\sp{\perp}\vee w)=(p\vee
 z\sp{\perp})\wedge(p\sp{\perp}\vee t\sp{\perp}),
\]
whence by Equations (\ref{eq:ecar02}) and (\ref{eq:ecar01}),
$c\dg=(p\vee z\sp{\perp})\wedge(p\sp{\perp}\vee t\sp{\perp})$.
By a similar calculation, $s\dg=(p\vee t\sp{\perp})\wedge
(p\sp{\perp}\vee z\sp{\perp})$.

 Part (ii) follows from Lemma \ref{lm:carrierofprod},
and (iii) follows from (i) and (ii).

To prove (iv), put $q:=(e-e\sp{2})\dg$, noting that by Corollary
\ref{co:zPropscor} (iii), $q=t\sp{\perp}\wedge z\sp{\perp}$. By
Definition \ref{df:j}, $j\sp{2}=p(e-e\sp{2})p+p\sp{\perp}(e-e\sp{2})p
\sp{\perp}$, and again it follows from \cite[Theorem 4.9 (v)]{FSynap}
and Lemma \ref{lm:carrierofsum} that
\begin{equation} \label{eq:j05}
j\,\dg=[p\wedge(p\sp{\perp}\vee q)]\vee[p\sp{\perp}\wedge(p\vee q)]
 =[p\wedge(p\sp{\perp}\vee q)]\vee v,
\end{equation}
where $v:=p\sp{\perp}\wedge(p\vee q)$. Now $pC(p\sp{\perp}\vee q)$
and $pCv$, whence
\begin{equation} \label{eq:j06}
[p\wedge(p\sp{\perp}\vee q)]\vee v=(p\vee v)\wedge(p\sp{\perp}\vee q
\vee v).
\end{equation} \label{eq:j07}
But $pCp\sp{\perp}$ and $pC(p\vee q)$, so
\begin{equation}
p\vee v=(p\vee p\sp{\perp})\wedge(p\vee p\vee q)=p\vee q.
\end{equation}
Furthermore, since $v\leq p\sp{\perp}$,
\begin{equation} \label{eq:j08}
p\sp{\perp}\vee q\vee v=p\sp{\perp}\vee q.
\end{equation}
Combining Equations (\ref{eq:j05})--(\ref{eq:j08}) and
the fact that $q=e\dg\wedge z\sp{\perp}$, we obtain (iv).

(v) Since ${s\dg}\sp{\perp}c\sp{2}={s\dg}\sp{\perp}(1-s\sp{2})=
{s\dg}\sp{\perp}-0={s\dg}\sp{\perp}$, we have ${s\dg}\sp{\perp}
\leq c\sp{2}\leq c$. Similarly, ${c\dg}\sp{\perp}s\sp{2}={c\dg}
\sp{\perp}(1-c\sp{2})={c\dg}\sp{\perp}-0={c\dg}\sp{\perp}$, whence
${c\dg}\sp{\perp}\leq s\sp{2}\leq s$.

(vi) Since $sCp$, we have $(s\dg)\sp{\perp}Cp$. Moreover, $(s\dg)
\sp{\perp}c\sp{2}=(s\dg)\sp{\perp}(1-s\sp{2})=(s\dg)\sp{\perp}$; by
(v), $b\dg\leq s\dg$, so $(s\dg)\sp{\perp}b=0$; and $(s\dg)\sp{\perp}
s\sp{2}p\sp{\perp}=0$; whence $(s\dg)\sp{\perp}e=(s\dg)\sp{\perp}
(c\sp{2}p+bk+s\sp{2}p\sp{\perp})=(s\dg)\sp{\perp}p=(s\dg)\sp{\perp}
\wedge p$. Similarly, $e(s\dg)\sp{\perp}=(pc\sp{2}+kb+p\sp{\perp}s
\sp{2})(s\dg)\sp{\perp}=p(s\dg)\sp{\perp}=p\wedge (s\dg)\sp{\perp}$,
so $(s\dg)\sp{\perp}Ce$ and $(s\dg)\sp{\perp}e=(s\dg)\sp{\perp}
\wedge e$ by Lemma \ref{lm:eCf}(ii). The proof of (vii) is similar.
\end{proof}

\begin{corollary}
If both $e$ and $e\sp{\perp}$ are projection free, then
$c\dg=s\dg=d\dg=1$.
\end{corollary}

A reasonable formula for $b\dg$ seems to be elusive; however, we
do have partial results as per the following lemma. (Also, see
Theorem \ref{th:commutatorineq} below.)

\begin{lemma} \label{lm:bdg}
Let $v:=kpk$. Then{\rm: (i)} $v$ is a projection, the symmetry $k$
exchanges $p$ and $v$, $bCv$, and $b\dg\leq(p\wedge v\sp{\perp})\vee(p
\sp{\perp}\wedge v)=(p\wedge v\sp{\perp})+(p\sp{\perp}\wedge v)$.
{\rm(ii)} If $p$ is an atom and $pe\not=ep$, then $p\perp v$ and
$b\dg=p\vee v=p+v$. {\rm(iii)} If $p$ is an atom and $pe\not=ep$,
then there exists $\beta\in\reals$ with $b=\beta b\dg$, $0<\beta
\leq 1$.
\end{lemma}

\begin{proof}
\setcounter{equation}{0}
(i) Obviously, $v$ is a projection and $k$ exchanges $p$ and $v$. By
parts (iii) and (iv) of Theorem \ref{th:CBSdecomp}, $bCk$ and $bCp$,
so $bCv$. Moreover, by Theorem \ref{th:CBSdecomp} (v), $bpkp=
bkp\sp{\perp}p=0$, whence, since $v=kpk\in P$,
\[
b\dg\leq((pkp)\dg)\sp{\perp}=(((pkp)\sp{2})\dg)\sp{\perp}=
((p(kpk)p)\dg)\sp{\perp}
\]
\begin{equation} \label{eq:bdg01}
=((pvp)\dg)\sp{\perp}=(p\wedge(p\sp{\perp}\vee v))\sp{\perp}=
 p\sp{\perp}\vee(p\wedge v\sp{\perp}).
\end{equation}
Starting with the observation that $bp\sp{\perp}kp\sp{\perp}=
bkpp\sp{\perp}=0$, and arguing as above, we deduce that
\begin{equation} \label{eq:bdg02}
b\dg\leq p\vee(p\sp{\perp}\wedge v).
\end{equation}
By (\ref{eq:bdg01}) and (\ref{eq:bdg02}),
\[
b\dg\leq[p\sp{\perp}\vee(p\wedge v\sp{\perp})]\wedge
[p\vee(p\sp{\perp}\wedge v)],
\] and using Theorem \ref{th:distributive} to simplify the
right side of the latter inequality, we obtain (i).

(ii) Suppose that $p$ is an atom and $pe\not=ep$. Since $k$
exchanges $p$ and $v$, it follows that $v$ is also an atom.
By Lemma \ref{lm:pCe}, $b\dg\not=0$, whence by (i), at least
one of the conditions $p\wedge v\sp{\perp}\not=0$ or $p\sp
{\perp}\wedge v\not=0$ must hold. Since $p$ and $v$ are atoms,
we have $p\perp v$ in either case, whence $p\wedge v\sp{\perp}=p$,
$p\sp{\perp}\wedge v=v$, so $p\perp v$ and by (i),
\begin{equation} \label{eq:vdg03}
0\not=b\dg\leq p\vee v=p+v.
\end{equation}

We claim that $p\leq b\dg$. Suppose not. Then, since $p$ is an
atom, $b\dg\wedge p=0$. Thus, as $bCp$, we have $b\dg Cp$, whence
$b\dg p=b\dg\wedge p=0$ and it follows that $bp=pb=0$. Consequently,
by Theorem \ref{th:CBSdecomp} (vi), $0=bpk=pep\sp{\perp}$, and
it follows from Lemma \ref{lm:offdiagzero} that $pCe$,
contradicting $pe\not=ep$. Therefore, $p\leq b\dg$.

We claim that $v\leq b\dg$. Suppose not. Then since $v$ is an
atom, $b\dg\wedge v=0$. Thus, as $bCv$, we have $b\dg Cv$, whence
$b\dg v=b\dg\wedge v=0$, and it follows that $bv=vb=0$. By
Theorem \ref{th:CBSdecomp} (vi), $bkp=bp\sp{\perp}k$, and we have
\begin{equation} \label{eq:vdg04}
0=bv=bkpk=bp\sp{\perp}k\sp{2}=bp\sp{\perp}=b(1-p)=b-bp,\text
 {\ so\ }b=bp.
\end{equation}
By Theorem \ref{th:CBSdecomp} (vi) again, $pep\sp{\perp}=bpk$
and $p\sp{\perp}ep=bkp$, whence by (\ref{eq:vdg04}),
\[
pep\sp{\perp}=bpk=bk \text{\ and therefore\ } p\sp{\perp}ep=bkp
 =(pep\sp{\perp})p=0,
\]
and again it follows from Lemma \ref{lm:offdiagzero} that $pCe$,
contradicting $pe\not=ep$. Therefore, $v\leq b\dg$.

Now we have $p,v\leq b\dg$, whereupon $p+v=p\vee v\leq b\dg$, which
together with (\ref{eq:vdg03}) yields $b\dg=p\vee v=p+v$.

(iii) Assume the hypotheses of (iii). By Theorem \ref{th:bProps} (i), $bp=pb=pbp$
and by Lemma \ref{lm:pAp,patom} (ii), (iii), $pb=bp=pbp=\beta p$ with $0\leq\beta
\leq 1$. Moreover, $bk=kb$ by Theorem \ref{th:CBSdecomp} (iii), and by Theorem
\ref{th:CBSdecomp} (v), $pbk=bpk=bkp^{\perp}$. Multiplying both sides of $b=
bp+bp^{\perp}$ by $k$, we obtain $kb=kbp+kbp^{\perp}=kbp+bpk=\beta kp+\beta pk
=\beta(kp+pk)$. Multiplying by $k$ again, we get $b=\beta(p+kpk)=\beta(p+v)=
\beta b\dg$ by (ii). Finally, since $pe\not=ep$, we have $b\not=0$ by Lemma
\ref{lm:pCe}, whence $0<\beta$.
\end{proof}

\section{Two commutators} \label{sc:Commutators}

The assumptions and notation set forth above remain in force. In this
section we study two candidates for a \emph{commutator projection}
for the pair $p\in P$, $e\in E$. Recall that in \cite[Definition 2.3]
{FJP2proj} the \emph{Marsden commutator} of two projections $p, q\in P$
is denoted and defined by
\[
[p,q]:=(p\vee q)\wedge(p\vee q\sp{\perp})\wedge(p\sp{\perp}\vee q)
\wedge(p\sp{\perp}\vee q\sp{\perp})
\]
and has the property that $pCq\Leftrightarrow[p,q]=0$. With this in mind,
for a projection $w\in P$ to be regarded as a \emph{commutator} for the
pair $p,e$, we shall require---at least---that $pCe\Leftrightarrow w=0$.
(Observe that the commutators defined in \cite[\S 5.1]{PP} satisfy the
dual condition that commutativity obtains iff the commutator equals $1$.)

The simplest candidate for a commutator projection for $p$ and $e$ is the
carrier projection $b\dg$ of the commutator effect $b$. By Lemma \ref
{lm:pCe}, $b\dg$ satisfies our basic condition $pCe\Leftrightarrow b
\dg=0$.

\begin{remark} \label{rm:bdgisMarsdenCom}
If it happens that $e\in P$, then $z=e$, $t=e\sp{\perp}$, and
$b=cs$, whence by Theorem 4.4 (iii),
\[
b\dg=(cs)\dg=(p\vee e)\wedge(p\vee e\sp{\perp})\wedge(p\sp{\perp}\vee e)
\wedge(p\sp{\perp}\vee e\sp{\perp})
\]
is the Marsden commutator $[p,e]$ of the pair of projections $p$ and $e$.
\end{remark}

Two projections are in so-called \emph{generic position} \cite[Definition 2.1]
{FJP2proj} iff their Marsden commutator is $1$; hence, by analogy, we say that
the projection $p$ and the effect $e$ are in \emph{generic position} iff $b\dg
=1$.

\begin{theorem}
Suppose that $p$ and $e$ are in generic position. Then{\rm:}
\begin{enumerate}
\item $(cs)\dg=c\dg=s\dg=1$.
\item $p\wedge z=p\wedge t=p\sp{\perp}\wedge z=p\sp{\perp}\wedge t=0$.
\item The symmetry $k$ in the CBS-decomposition of $e$ with respect to
 $p$ exchanges the projections $p$ and $p\sp{\perp}$.
\end{enumerate}
\end{theorem}

\begin{proof}
Assume that $p$ and $e$ are in generic position, i.e., $b\dg=1$.
Since $b\sp{2}=c\sp{2}s\sp{2}-d\sp{2}\leq c\sp{2}s\sp{2}$, it
follows that $1=b\dg=(b\sp{2})\dg\leq(c\sp{2}s\sp{2})\dg=(cs)\dg=
c\dg s\dg$, proving (i). Part (ii) follows from (i), Theorem
\ref{th:ecarcs} (iii), and De\,Morgan. By Theorem \ref{th:CBSdecomp}
(v), $pk=kp\sp{\perp}$, whence $kpk=p\sp{\perp}$, proving (iii).
\end{proof}

There are two possible shortcomings of $b\dg$ as a commutator projection
for the pair $p$ and $e$: First, although $p$ commutes with $b\dg$, in
general, $e$ fails to commute with $b\dg$ (see Example \ref{ex:bdg<[p,e]}
below). Second, as we mentioned earlier, obtaining a perspicuous formula
for $b\dg$ in terms of $e\dg, (e\sp{\perp})\dg, p, c\dg, s\dg, z$, $t$,
and $k$ seems to offer a challenge.

In the following definition, we shall extend the Marsden commutator
for two projections to a commutator $[F]$ for a finite set $F\subseteq P$
of projections. We note that this definition is dual to \cite
[Definition 5.1.4]{PP}, i.e., suprema and infima have been interchanged.

\begin{definition} \label{df:[F]}
Suppose that $F=\{w_1,w_2,\ldots,w_n\}\subseteq P$ is a finite set of
projections. For any $w\in P$, let us write $w^1:=w$ and $w^{-1}:=
w^{\perp}$. Further, let $D:=\{1,-1\}$.
Then, as $d=(d_1,d_2,\ldots,d_n)$ runs through $D\sp{n}$, the \emph
{commutator} of the set $F$ is denoted and defined by
\[
[F]:=\bigwedge_{d\in D^n}(w_1^{d_1}\vee w_2^{d_2}\vee\cdots
 \vee w_n^{d_n})\in P.
\]
Also, we define $[\emptyset]:=0$.
\end{definition}
\noindent Clearly, $[\{w\sb{1}\}]=0$. Also, if $F=\{w\sb{1},w\sb{2}\}$,
then
\[
[F]=(w_1\vee w_2)\wedge(w_1^{\perp}\vee w_2)\wedge(w_1\vee w_2^{\perp})
 \wedge(w_1^{\perp}\vee w_2^{\perp})
\]
is the Marsden commutator of $w\sb{1}$ and $w\sb{2}$. We note that
if the special projections $0$ or $1$ are present in $F$, then
$[F\setminus\{0,1\}]=[F]$.

\begin{remark} \label{rm:qC[F]}
Suppose that $F$ is a finite subset of $P$, $q\in P$, and $qCw$
for every $w\in F$. Then since commutativity is preserved under
formation of orthocomplements, finite suprema, and finite infima,
it follows that $qC[F]$.
\end{remark}

\begin{remark} \label{rm:Freplacement}
If $F$ is a finite subset of $P$, it is obvious that $[F]$ is
unchanged if one of the projections in $F$ is replaced by its
orthocomplement. As a consequence, if both $w\in F$ and $w
\sp{\perp}\in F$, then $w\sp{\perp}$ can be omitted from $F$
without affecting the value of $[F]$.
\end{remark}

By dualizing \cite[Theorem 5.1.5 and Prop. 5.1.8]{PP}, we obtain the
following characterization of $[F]$.

\begin{lemma} \label{lm:[F]Props}
Let $F\subseteq P$ be a finite set of projections and put $r:=[F]$.
Then{\rm:}
\begin{enumerate}
\item $w\in F\Rightarrow rCw$.
\item The projections in the set $\{w\wedge r\sp{\perp}:w\in F\}$
 commute pairwise.
\item $r$ is the smallest projection that satisfies {\rm(i)} and
 {\rm(ii)}.
\item $r=0$ iff the projections in the set $F$ commute pairwise.
\end{enumerate}
\end{lemma}

Now, by dualizing \cite[Def. 5.1.6]{PP}, we shall extend Definition
\ref{df:[F]} to arbitrary countable subsets $W$ of $P$. (By \emph
{countable}, we mean finite or countably infinite.) However our
definition will require that the OML $P$ is \emph{$\sigma$-complete},
i.e., that every countable subset of $P$ has a supremum (whence also
an infimum) in $P$. It is known that $P$ is $\sigma$-complete iff it is
$\sigma$-orthocomplete, i.e., iff every countable and pairwise
orthogonal subset of $P$ has a supremum in $P$ \cite[Corollary 3.4]
{JencaCB}. According to the discussion in \cite[\S6]{FSynap}, every
\emph{generalized Hermitian algebra} \cite{FPGHA, FPSpin, FPRegGHA}
is a synaptic algebra with a $\sigma$-complete projection lattice. For
instance, the self-adjoint part of a von Neumann algebra has a
$\sigma$-complete (and in fact, a complete) projection lattice. Thus
we make the following assumption.

\begin{assumption} \label{as:sigmacomplete}
Henceforth in this section, we assume that the OML $P$ is $\sigma$-orthocomplete;
hence $\sigma$-complete.
\end{assumption}

\begin{remarks}
Since there are only countably many finite subsets of a countable
set, the supremum in the following definition exists. Also, if
$W\subseteq P$ is a finite set, then (as is easily seen) $[W]=
\bigvee\{\,[F]:F\subseteq W\}$. Therefore, the following definition
provides a true generalization of $[F]$ for a finite set
$F\subseteq P$.
\end{remarks}

\begin{definition} \label{df:[S]}
For an arbitrary countable subset $W\subseteq P$, the \emph{commutator}
of $W$ is denoted and defined by
\[
[W]=\bigvee\{\,[F]:F\subseteq W\text{\ and\ }F\text{\ is finite}\}.
\]
\end{definition}

\begin{remark} \label{rm:vC[S]}
Suppose that $W$ is a countable subset of $P$, $q\in P$, and
$qCw$ for every $w\in W$. Then since commutativity is preserved
under formation of arbitrary existing suprema, it follows from
Remark \ref{rm:qC[F]} that $qC[W]$.
\end{remark}

\begin{remarks} \label{rm:Wreplacement}
If $W$ is a countable subset of $P$, then as a consequence of
Remark \ref{rm:Freplacement}, $[W]$ is unchanged if one of the
projections in $W$ is replaced by its orthocomplement. As a
consequence, if both $w\in W$ and $w\sp{\perp}\in W$, then $w
\sp{\perp}$ can be omitted from $W$ without affecting the value
of $[W]$.
\end{remarks}

By dualizing \cite[Theorem 5.1.7 and Prop. 5.1.8]{PP}, we obtain the
following characterization of $[W]$.

\begin{theorem} \label{th:[W]}
If $W\subseteq P$, $W$ is countable, and $r:=[W]$, then{\rm:}
\begin{enumerate}
\item $w\in W\Rightarrow rCw$.
\item The projections in the set $\{w\wedge r\sp{\perp}:w\in W\}$
 commute pairwise.
\item $r$ is the smallest projection with properties {\rm(i)} and
{\rm(ii)}.
\item $r=0$ iff the projections in the set $W$ commute pairwise.
\end{enumerate}
\end{theorem}

Using Assumption \ref{as:sigmacomplete}, Definition \ref{df:[S]},
and the notion of a rational spectral resolution, we are now in a
position to define an alternative $[p,e]$ to $b\dg$ as a commutator
for the pair $p,e$.

\begin{definition} \label{df:[pe]}
For $p\in P$ and $e\in E$, the \emph{commutator}
of the pair $p,e$ is denoted and defined by
\[
[p,e]:=\left[\{p\}\cup\{p\sb{e,\mu}:\mu\in\rationals\}\right].
\]
\end{definition}

As we shall see in Corollary \ref{co:equalityofcoms} (ii) below,
no notational conflict with the Marsden commutator of two projections
in \cite{FJP2proj} will result from the use of the notation $[p,e]$ in
Definition \ref{df:[pe]}.

We note that, in Definition \ref{df:[pe]}, only the \emph{set} of
projections in the rational spectral resolution of $e$ is
involved---the labeling of these projections by rational numbers
plays no role in the computation of $[p,e]$.

In the following theorem, which characterizes $[p,e]$, recall that
by Lemma \ref{lm:eCf} (ii), if $q\in P$ and $qCe$, then $q\sp
{\perp}e=eq\sp{\perp}=e\wedge q\sp{\perp}$, the infimum of $e$
and $q\sp{\perp}$ in $E$.

\begin{theorem} \label{th:Characterize[p,e]}
If $p\in P$ and $e\in E$, then $[p,e]$ is the smallest projection
$q\in P$ such that $qCp$, $qCe$, and $(p\wedge q\sp{\perp})C
(e\wedge q\sp{\perp})$.
\end{theorem}

\begin{proof}
Put $W:=\{p\}\cup\{p\sb{e,\mu}:\mu\in\rationals\}$ and $r:=[p,e]=
[W]$. By Theorem \ref{th:[W]}, we have: (i) $w\in W\Rightarrow rCw$.
(ii) The projections in the set $\{w\wedge r\sp{\perp}:w\in W\}$
commute pairwise. (iii) $r$ is the smallest projection with properties
{\rm(i)} and {\rm(ii)}.

We claim that (iv) $rCp$, (v) $rCe$, and (vi) $(p\wedge r\sp
{\perp})C(e\wedge r\sp{\perp})$. Indeed, since $p\in W$, (i)
implies that $rCp$. Also by (i), for every $\mu\in\rationals$,
$rCp\sb{e,\mu}$, whence by Remark \ref{rk:RatSpecCommute}, $rCe$.
Moreover, for every $\mu\in\rationals$, we have both $p\in W$ and
$p\sb{e,\mu}\in W$, whence $(p\wedge r\sp{\perp})C(p\sb{e,\mu}
\wedge r\sp{\perp})$ by (ii). But by Lemma \ref{lm:SRofqaq}, $(p\sb{e,
\lambda}\wedge r\sp{\perp})\sb{\lambda\in\reals}$ is the spectral
resolution of $e\wedge r\sp{\perp}$ as calculated in $r\sp{\perp}Ar
\sp{\perp}$; hence by Remark \ref{rk:RatSpecCommute} again,
$p\wedge r\sp{\perp}$ commutes with $e\wedge r\sp{\perp}$ in
$r\sp{\perp}Ar\sp{\perp}$, and therefore also in $A$. Thus we have
(iv), (v), and (vi).

Now assume that $v\in P$, $vCp$, $vCe$, and $(p\wedge v\sp{\perp})C
(e\wedge v\sp{\perp})$. We have to prove that $r\leq v$. By (iii)
it will be sufficient to show that (i$\sp{\,\prime}$) $w\in W
\Rightarrow vCw$ and (ii$\sp{\,\prime}$) the projections in the set
$\{w\wedge v\sp{\perp}:w\in W\}$ commute pairwise. To prove
(i$\sp{\,\prime}$), suppose $w\in W$. If $w=p$, we have $vCw$, so
we can assume that $w=p\sb{e,\mu}$ for some $\mu\in\rationals$. But
since $vCe$, it follows that $vCp\sb{e,\mu}$, and we have
(i$\sp{\,\prime}$).

To prove (ii$\sp{\,\prime}$), suppose that $w,q\in W$. First we
consider the case $w=p$ and $q=p\sb{e,\nu}$ with $\nu\in\rationals$.
Since $vCe$, we have $eCv\sp{\perp}$, whence by Lemma \ref{lm:SRofqaq}
(ii), $(p\sb{e,\lambda}\wedge v\sp{\perp})\sb{\lambda\in\reals}$ is
the spectral resolution of $e\wedge v\sp{\perp}$ as calculated in
$v\sp{\perp}Av\sp{\perp}$. By hypothesis, $(p\wedge v\sp{\perp})C
(e\wedge v\sp{\perp})$, and it follows that $(p\wedge v\sp{\perp})C
(p\sb{e,\nu}\wedge v\sp{\perp})$.  This reduces our argument
to the case $w=p\sb{e,\mu}$ and $q=p\sb{e,\nu}$ with $\mu,\nu\in
\rationals$. But, the projections in a spectral resolution commute
pairwise, whence $(p\sb{e,\mu}\wedge v\sp{\perp})C(p\sb{e,\nu}\wedge v
\sp{\perp})$, proving (ii$\sp{\,\prime}$).
\end{proof}

By the following corollary to Theorem \ref{th:Characterize[p,e]},
$[p,e]$ qualifies as a commutator of $p$ and $e$.

\begin{corollary} \label{co:[p,e]}
If $p\in P$ and $e\in E$, then $pCe\Leftrightarrow [p,e]=0$.
\end{corollary}

\begin{proof}
If $pCe$, then $0Cp$, $0Ce$, and $(p\wedge 0\sp{\perp})C(e\wedge 0
\sp{\perp})$, whence $[p,e]\leq 0$, i.e., $[p,e]=0$. Conversely,
if $[p,e]=0$, then $(p\wedge 0\sp{\perp})C(e\wedge 0\sp{\perp})$,
i.e., $pCe$.
\end{proof}

\begin{lemma} \label{lm:randCBS}
Let $r:=[p,e]$. Then{\rm: (i)} In the CBS-decomposition of $e$ with
respect to $p$, we have $rCp$, $rCe$, $rCc$, $rCs$, $rCj$, $rCb$
and $rCk$. {\rm (ii)} If $q\in P$, $qCp$, and $qCe$, then $qCr$.
\end{lemma}

\begin{proof}
(i) By Theorem \ref{th:Characterize[p,e]}, $rCp$ and $rCe$. Since $c=
(pep+p\sp{\perp}e\sp{\perp}p\sp{\perp})\sp{1/2}\in CC(pep+p\sp
{\perp}e\sp{\perp}p\sp{\perp} )$, it follows that $rCc$, and similarly,
$rCs$. Also, $rC(e-e\sp{2})$, and because  $j=(p(e-e\sp{2})p+p\sp
{\perp}(e-e\sp{2})p\sp{\perp})\sp{1/2}$, it follows that $rCj$.
Therefore, as $b=(c\sp{2}s\sp{2}-j\sp{2})\sp{1/2}$, we have $rCb$.
Finally, $rCk$ follows from $k\in CC(pep\sp{\perp}+ p\sp{\perp}ep)$.

(ii) Suppose $q\in P$, $qCp$, and $qCe$. Then $qCp\sb{e,\mu}$ for
all $\mu\in\rationals$, whence $qCr$ by Definition \ref{df:[pe]} and
Remark \ref{rm:vC[S]}.
\end{proof}

\begin{theorem} \label{th:cominqAq}
Let $q\in P$, suppose that $qCp$ and $qCe$, let $r:=[p,e]$ and let
$v\in P[0,q]$ be the commutator $[p\sb{q},e\sb{q}]\sb{qAq}$ of $p
\sb{q}=pq$ and $e\sb{q}=eq$ as calculated in $qAq$. Then $qCr$,
$pCr$, $eCr$, $qCv$, $pCv$, $eCv$, and $v=r\sb{q}=rq=qr=q\wedge r$.
\end{theorem}

\begin{proof}
Since $qCp$ and $qCe$, we have $qCr$ by Lemma \ref{lm:randCBS} (ii).
Also, $pCr$ and $eCr$ by Lemma \ref{lm:randCBS} (i). As $v\in P[0,q]$,
we have $v=qv=vq$ and $v\sp{\perp\sb{q}}=qv\sp{\perp}=v\sp{\perp}q$.
Thus, by Theorem \ref{th:Characterize[p,e]} applied to $p\sb{q}$ and
$e\sb{q}$ in the synaptic algebra $qAq$, we infer that $v$ is the
smallest projection in $P[0,q]$ such that
\[
{\rm(i)\ } vC(pq), {\ \rm(ii)\ } vC(eq),\text{\ and\ }{\rm(iii)\ } ((pq)
 \wedge(v\sp{\perp}q))C((eq)\wedge(v\sp{\perp}q)).
\]
Since $vC(pq)$ and $qCp$, it follows that $pv=p(qv)=(pq)v=v(pq)=v(qp)=
(vq)p=vp$, whence $pCv$. Likewise, since $vC(eq)$ and $qCe$, it follows
that $ev=e(qv)=(eq)v=v(eq)=v(qe)=(vq)e=ve$, whence $eCv$. Thus the three
elements $p$, $q$, and $v\sp{\perp}$ commute in pairs, and so do the
three elements $e$, $q$ and $v\sp{\perp}$. Consequently, $p\wedge(v\vee q
\sp{\perp})\sp{\perp}=p\wedge v\sp{\perp}\wedge q=pv\sp{\perp}q=pqv
\sp{\perp}q=(pq)\wedge(v\sp{\perp}q)$, similarly $e\wedge(v\vee q
\sp{\perp})=(eq)\wedge(v\sp{\perp}q)$, and we can rewrite (iii) as
$(p\wedge(v\vee q\sp{\perp})\sp{\perp})C(e\wedge(v\vee q\sp{\perp})
\sp{\perp})$. Furthermore, $(v\vee q\sp{\perp})Cp$ and $(v\vee q
\sp{\perp})Ce$, and it follows from Theorem \ref{th:Characterize[p,e]}
that $r=[p,e]\leq v\vee q\sp{\perp}$. Therefore, $r\sb{q}=rq=r\wedge q
\leq(v\vee q\sp{\perp})\wedge q=v\wedge q=v$.

To complete the proof, we have to show that $v\leq r\sb{q}$, i.e., that
$v\leq rq$.  Since $rCp$, $rCq$, and $qCp$, we have $(rq)C(pq)$. Likewise,
since $rCe$, $rCq$, and $qCe$, we have $(rq)C(eq)$. Thus, with $v$ replaced
by $rq$, conditions (i) and (ii) hold; hence, to prove that $v\leq rq$,
it will be sufficient to prove that condition (iii) holds with $v$
replaced by $rq$, i.e., that $((pq)\wedge((rq)\sp{\perp}q))C
((eq)\wedge((rq)\sp{\perp}q))$. Since $rCq$, we have $(rq)\sp{\perp}q=
(r\wedge q)\sp{\perp}\wedge q=(r\sp{\perp}\vee q\sp{\perp})\wedge q=
r\sp{\perp}\wedge q=qr\sp{\perp}$. Thus, as $p$, $q$, and $r$ commute
pairwise, we have $(pq)\wedge((rq)\sp{\perp}q)=(pq)\wedge(qr\sp{\perp})
=pqr\sp{\perp}$. Likewise, as $e$, $q$, and $r$ commute
pairwise, we deduce that $(eq)\wedge((rq)\sp{\perp}q)=eqr\sp{\perp}$.
Thus, it will be sufficient to show that $(pqr\sp{\perp})C(eqr\sp
{\perp})$. By Theorem \ref{th:Characterize[p,e]}, $(pr\sp{\perp})
C(er\sp{\perp})$; hence, as $qCp$, $qCr\sp{\perp}$, and $qCe$,
we have
\[
(pqr\sp{\perp})(eqr\sp{\perp})=q(pr\sp{\perp})(er\sp{\perp}q)=
q(pr\sp{\perp})(er\sp{\perp})q
\]
\[
=q(er\sp{\perp})(pr\sp{\perp})q=
(eqr\sp{\perp})(pqr\sp{\perp}),
\]
so $(pqr\sp{\perp})C(eqr\sp{\perp})$.
\end{proof}

\begin{theorem} \label{th:rProps}
Let $r:=[p,e]$. Then{\rm: (i)} $p\sb{r\sp{\perp}}Ce\sb{r\sp{\perp}}$.
{\rm(ii)} $b\sb{r\sp{\perp}}=0$ and $e\sb{r\sp{\perp}}=c\sb{r\sp{\perp}}
\sp{\,2}p\sb{r\sp{\perp}}+s\sb{r\sp{\perp}}\sp{\,2}(p\sb{r\sp{\perp}})
\sp{\,\perp\sb{r\sp{\perp}}}$.
\end{theorem}

\begin{proof}
(i) By Theorem \ref{th:Characterize[p,e]}, $(p\wedge r\sp{\perp})C
(e\wedge r\sp{\perp})$, proving (i).

(ii) By Theorem \ref{th:esubq} (iii) with $q:=r\sp{\perp}$, the
CBS-decomposition of $e\sb{r\sp{\perp}}$ with respect to $p\sb
{r\sp{\perp}}$ in $p\sp{\perp}Ap\sp{\perp}$ is $e\sb{r\sp{\perp}}
=c\sb{r\sp{\perp}}\sp{\,2}p\sb{r\sp{\perp}}+b\sb{r\sp{\perp}}k
\sb{r\sp{\perp}}+s\sb{r\sp{\perp}}\sp{\,2}(p\sb{r\sp{\perp}})
\sp{\,\perp\sb{r\sp{\perp}}}$. But by (i) and Lemma \ref
{lm:pCe}, $b\sb{r\sp{\perp}}=0$.
\end{proof}

\begin{theorem} \label{th:commutatorineq}
$b\leq b\dg\leq[p,e]\leq c\dg\wedge s\dg=(cs)\dg=c\dg s\dg$.
\end{theorem}

\begin{proof}
Put $r:=[p,e]$. By Theorem \ref{th:rProps} (ii), $br\sp{\perp}=
b\sb{r\sp{\perp}}=0$, so $b\leq r$, and therefore $b\leq b\dg
\leq r=[p,e]$.

Put $q:=s\dg$. Then by Theorem \ref{th:ecarcs} (vi), $qCp$,
$qCe$, and $p\wedge q\sp{\perp}=e\wedge q\sp{\perp}$, so
$(p\wedge q\sp{\perp})C(e\wedge q\sp{\perp})$. Therefore,
by Theorem \ref{th:Characterize[p,e]}, $[p,e]\leq q=s\dg$.
A similar argument using Theorem \ref{th:ecarcs} (vii) shows
that $[p,e]\leq c\dg$, and we have $[p,e]\leq c\dg\wedge s\dg
=(cs)\dg=c\dg s\dg$ (Theorem \ref{th:ecarcs} (ii)).
\end{proof}

Using the fact that $b\dg\leq [p,e]$, we obtain the following
alternative characterization of $[p,e]$.

\begin{theorem} \label{th:altchar[p,e]}
$[p,e]$ is the smallest projection $v$ such that $vCp$, $vCe$,
and $b\dg\leq v$.
\end{theorem}

\begin{proof}
Put $r:=[p,e]$. By Lemma \ref{lm:randCBS} (i), $rCp$ and $rCe$ and
by Theorem \ref{th:commutatorineq}, $b\dg\leq r$.  Suppose that $v\in
P$, $vCp$, $vCe$, and $b\dg\leq v$. We have to prove that $r\leq v$.
We have $b\leq b\dg\leq v$, whence $bv\sp{\perp}=v\sp{\perp}b=0$.
Moreover, as $vCp$, $vCe$, and $c\sp{2}=pep+p\sp{\perp}e\sp{\perp}p
\sp{\perp}$, it follows that $vCc\sp{2}$. Likewise, $vCs\sp{2}$,
whence $v\sp{\perp}$ commutes with both $e$ and $p$, whereas both
$v\sp{\perp}$ and $p$ commute with $c\sp{2}$, $p$, $s\sp{2}$, and
$p\sp{\perp}$. Therefore, by the CBS-decomposition of $e$ with respect to $p$,
\[
ev\sp{\perp}=v\sp{\perp}e=v\sp{\perp}c\sp{2}p+v\sp{\perp}bk+v\sp{\perp}
s\sp{2}p\sp{\perp}=v\sp{\perp}c\sp{2}p+v\sp{\perp}s\sp{2}p\sp{\perp},
\]
and since $pv\sp{\perp}$ commutes with both $v\sp{\perp}c\sp{2}p$ and
$v\sp{\perp}s\sp{2}p\sp{\perp}$ it follows that $pv\sp{\perp}$
commutes with $ev\sp{\perp}$, i.e., $(p\wedge v\sp{\perp})C(e\wedge
v\sp{\perp})$. Consequently, by Theorem \ref{th:Characterize[p,e]},
$r\leq v$.
\end{proof}

\begin{corollary} \label{co:equalityofcoms}
{\rm(i)} $b\dg=[p,e]$ iff $eCb\dg$. {\rm(ii)} If $e\in P$, then
$b\dg=[p,e]$ is the Marsden commutator of the two projections $p$
and $e$.
\end{corollary}

\begin{proof}
(i) If $eCb\dg$, then both $b\dg Cp$ and $b\dg Ce$ hold, whence
$b\dg=[p,e]$ by Theorem \ref{th:altchar[p,e]}. Conversely, by
Theorem \ref{th:altchar[p,e]} again, if $b\dg=[p,e]$, then
$eCb\dg$.

(ii) Suppose that $e\in P$. Temporarily denoting the Marsden
commutator of $p$ and $e$ by $[p,e]\sb{M}$, we infer from Remark
\ref{rm:bdgisMarsdenCom} that $b\dg=(cs)\dg=[p,e]\sb{M}$. By
\cite[Theorem 3.8 (vi)]{FJP2proj} with $q:=e$, we have $eC[p,e]
\sb{M}$, whence $eCb\dg$. Therefore by (i), $[p,e]\sb{M}=b\dg=[p,e]$.
\end{proof}

The projection $p$ and the effect $e$ are said to be \emph{totally
noncompatible} iff $[p,e]=1$. If $p$ and $e$ are in generic position,
then $1=b\dg\leq [p,e]$, and it follows that $p$ and $e$ are totally
noncompatible.

\begin{lemma} \label{lm:totnoncomp}
Suppose that $p$ and $e$ are totally noncompatible.  Then{\rm: (i)}
If $v\in P$, $vCp$, $vCe$, and $(p\wedge v)C(e\wedge v)$, then
$v=0$. {\rm(ii)} $c\dg=s\dg=(cs)\dg=1$. {\rm(iii)} $p\wedge z=
p\wedge t=p\sp{\perp}\wedge z=p\sp{\perp}\wedge t=0$.
\end{lemma}

\begin{proof}
By hypothesis, $[p,e]=1$. Part (i) follows from Theorem \ref
{th:Characterize[p,e]}, part (ii) follows from Theorem \ref
{th:commutatorineq}, and (iii) is a consequence of (ii),
parts (ii) and (iii) of Theorem \ref{th:ecarcs}, and De\,Morgan.
\end{proof}

The following example shows that it is possible to
have $p$ and $e$ totally noncompatible (i.e., $[p,e]=1$), where $p$
and $e$ are not in generic position (i.e., $b\dg<[p,e]=1$).

\begin{example} \label{ex:bdg<[p,e]}
Let $\reals\sp{3}$ be organized as usual into a $3$-dimensional real
Hilbert space and let $A$ be the synaptic algebra of all self-adjoint
linear operators on $\reals\sp{3}$. Let $p\sb{1}, p\sb{2}, p\sb{3}$,
and $p$ be the (orthogonal) projections onto the one-dimensional
subspaces $\{(\alpha,0,0):\alpha\in\reals\}$, $\{(0,\beta,0):\beta\in
\reals\}$, $\{(0,0,\gamma):\gamma\in\reals\}$, and $\{(\xi,\xi,\xi):\xi
\in\reals\}$, respectively. Put $e:=\frac{1}{4}p\sb{1}+\frac{1}{2}p\sb{2}
+\frac{3}{4}p\sb{3}$, noting that $e$ is an effect in $A$ and that the
set of projections in the spectral resolution of $e$ is $\{0,p\sb{1},
p\sb{1}+p\sb{2},1\}$. As observed above, in forming $[p,e]$, we can omit
the projections $0$ and $1$, whence $[p,e]=[\{p,p\sb{1},p\sb{1}+p\sb{2}\}]$.
Also, $p\sb{1}+p\sb{2}=p\sb{3}\sp{\perp}$, so by Remark \ref{rm:Freplacement},
$[p,e]=[\{p,p\sb{1},p\sb{3}\}]$.

We observe that each of the atoms $p\sb{1}$, $p\sb{3}$, and $p\sb{1}
\sp{\perp}\wedge p\sb{3}\sp{\perp}=p\sb{2}$ is disjoint from both
$p$ and $p\sp{\perp}$, whence with the notation of Definition
\ref{df:[F]}, $p\sp{d\sb{1}}\wedge p\sb{1}\sp{d\sb{2}}\wedge p
\sb{3}\sp{d\sb{3}}=0$ for all $d\in D\sp{3}$. Therefore, by De\,Morgan,
\[
[p,e]=\bigwedge\sb{d\in D\sp{3}}(p\sp{d\sb{1}}\vee p\sb{1}\sp{d\sb{2}}
 \vee p\sb{3}\sp{d\sb{3}})=1,
\]
i.e. $p$ and $e$ are totally noncompatible. In particular $pe\not=ep$.
As $p$ is an atom in $P$, so is $v:=kpk$. Thus by Lemma \ref{lm:bdg}
(ii), $p\perp v$ and $b\dg=p\vee v=p+v$ so $b\dg$ is a two-dimensional
(i.e., rank 2) projection, and therefore $b\dg\not=1=[p,e]$. As a
consequence (Corollary \ref{co:equalityofcoms}), $b\dg$ does not
commute with $e$.
\end{example}

\begin{theorem} \label{th:totnoncomp}
Let $r:=[p,e]$. Then the projection $p\sb{r}=pr=rp=p\wedge r$ and the
effect $e\sb{r}=er=re=e\wedge r$ are totally noncompatible in $rAr$.
\end{theorem}

\begin{proof}
By Lemma \ref{lm:randCBS} (i), $rCp$ and $rCe$, whence, putting $q:=r$
in Theorem \ref{th:cominqAq}, we find that the commutator $[p\sb{r},
e\sb{r}]\sb{rAr}$ as calculated in $rAr$ is given by $[p\sb{r},
e\sb{r}]\sb{rAr}=r\wedge r=r$. But $r$ is the unit element in $rAr$,
proving the theorem.
\end{proof}

By Theorem \ref{th:rProps} (i) and Theorem \ref{th:totnoncomp}, the
projection $p=p\sb{r}+p\sb{r\sp{\perp}}$ and the effect $e=e\sb{r}+
e\sb{r\sp{\perp}}$ are decomposed into components $p\sb{r}$, $e
\sb{r}$ that are totally noncompatible in $rAr$ and components
$p\sb{r\sp{\perp}},e\sb{r\sp{\perp}}$ that commute in $r\sp{\perp}
Ar\sp{\perp}$.

\section{An application of CBS-decomposition} \label{sc:Appl}

If $A$ is the synaptic algebra of all self-adjoint operators on a
complex Hilbert space, then (transcribed to our current notation),
T. Morland and S. Gudder prove that, if $e\in E$ and $p$ is an atom
in $P$, then $e\wedge p\sp{\perp}$ exists in $E$ \cite[Lemma 3.8]{MG}.
Morland and Gudder's proof uses the Hilbert-space inner product and
the Schwarz inequality, and thus is not available for our more general
synaptic algebra. However, using the CBS-decomposition we generalize
\cite[Lemma 3.8]{MG} to our present setting in Theorem \ref{th:MGL3.8}
below.

\begin{assumptions}
In this section the notation and assumptions of Sections
\ref{sc:Basic}--\ref{sc:Commutators} remain in force. In addition,
we assume that {\rm(i)} $p$ is an atom in $P$ and {\rm(ii)} as per
Lemma \ref{lm:pAp,patom}, $pep=\alpha p$ with $\alpha\in\reals$,
$0\leq\alpha\leq 1$.
\end{assumptions}

\begin{definition} \label{df:a&y}
If $\alpha>0$, we define:
(1) $a:=\alpha\sp{-1}bk=\alpha\sp{-1}(pep\sp{\perp}+p\sp{\perp}ep)\in A$.
(2) $y:=p\sp{\perp}(1-a)\in R$. (3) $y\sp{\ast}:=(1-a)p\sp{\perp}\in R$.
\end{definition}

\noindent Provided that $\alpha>0$, the mapping $f\mapsto yfy\sp{\ast}$ for
$f\in A$ is the composition of the quadratic mappings $f\mapsto g:=(1-a)f(1-a)$
and $g\mapsto p\sp{\perp}gp\sp{\perp}$, whence it is a linear and order-preserving
mapping on $A$.

We omit the straightforward computational proofs of the next three lemmas.

\begin{lemma} Suppose that $\alpha>0$. Then{\rm:}
{\rm(i)} $ap=p\sp{\perp}a=p\sp{\perp}ap=\alpha\sp{-1}p\sp{\perp}ep$,
$ap\sp{\perp}=pa=pap\sp{\perp}=\alpha\sp{-1}pep\sp{\perp}$, and $ap+pa=a$.
{\rm(ii)} $(ap)\sp{2}=0$. {\rm(iii)} $(pa)\sp{2}=0$. {\rm(iv)} $y=p\sp{\perp}
-ap$ and $y\sp{\ast}=p\sp{\perp}-pa$.
\end{lemma}

\begin{lemma} \label{lm:MB01}
Suppose that $\alpha>0$. Then{\rm:}
{\rm(i)} The CBS-decomposition of $e$ with respect to the atom $p$ is
$e=\alpha p+\alpha a+s\sp{2}p\sp{\perp}$. {\rm(ii)} $\alpha\sp{2}a\sp{2}
=b\sp{2}$. {\rm(iii)} $epe=\alpha\sp{2}p+\alpha\sp{2}a+b\sp{2}p\sp{\perp}$.
{\rm(iv)} $e-\alpha\sp{-1}epe=(s\sp{2}-\alpha\sp{-1}b\sp{2})p\sp{\perp}=
p\sp{\perp}(s\sp{2}-\alpha\sp{-1}b\sp{2})$.
\end{lemma}

\begin{lemma} \label{lm:ygystar}
Suppose that $f\in A$ and $\alpha>0$. Then{\rm:}
{\rm(i)} $0\leq f\leq p\sp{\perp}\Rightarrow yfy\sp{\ast}=f$.
{\rm(ii)} $0\leq yey\sp{\ast}=(s\sp{2}-\alpha\sp{-1}b\sp{2})p\sp{\perp}=e
 -\alpha\sp{-1}epe$.
\end{lemma}

\begin{theorem} \label{th:MGL3.8}
The infimum $e\wedge p\sp{\perp}$ exists in $E$. In fact, if $\alpha=0$,
then $e\wedge p\sp{\perp}=e$, and if $\alpha>0$, then $e\wedge p
\sp{\perp}=(s\sp{2}-\alpha\sp{-1}b\sp{2})p\sp{\perp}=e-\alpha\sp{-1}epe$.
\end{theorem}

\begin{proof}
If $\alpha=0$, then $pep=0$, and as $0\leq e$ it follows that $pe=ep=0$
(\cite[Axiom SA4]{FSynap}), whence $e\leq p\sp{\perp}$, so $e=e\wedge p
\sp{\perp}$.

Now suppose that $\alpha>0$. By Lemma \ref{lm:ygystar} (ii), $0\leq(s\sp{2}
-\alpha\sp{-1}b\sp{2})p\sp{\perp}=e-\alpha\sp{-1}epe$. Since $0\leq epe$,
we also have $e-\alpha\sp{-1}epe\leq e\leq 1$, so $e-\alpha\sp{-1}epe\in E$.
Moreover, $e-\alpha\sp{-1}epe=(s\sp{2}-\alpha\sp{-1}b\sp{2})p\sp{\perp}
\leq p\sp{\perp}$. Suppose that $f\in E$ with $f\leq e,p\sp{\perp}$.  Then
by Lemma \ref{lm:ygystar} (i), $0\leq y(e-f)y\sp{\ast}=yey\sp{\ast}-yfy
\sp{\ast}=(s\sp{2}-\alpha\sp{-1}b\sp{2})p\sp{\perp}-f$, whence $f\leq
(s\sp{2}-\alpha\sp{-1}b\sp{2})p\sp{\perp}$, and it follows that $e\wedge p
\sp{\perp}=(s\sp{2}-\alpha\sp{-1}b\sp{2})p\sp{\perp}=e-\alpha\sp{-1}epe$.
\end{proof}

\begin{corollary} [{Cf. \cite[Corollaries 3.9 and 3.10]{MG}}] \label{co:MG3.9}
\
\begin{enumerate}
\item If $p\sb{1}, p\sb{2}, ..., p\sb{n}$ is a finite sequence of mutually
orthogonal atoms in $P$, then $e\wedge(p\sb{1}\vee p\sb{2}\vee\cdots\vee p
\sb{n})\sp{\perp}$ exists in $E$.
\item Suppose that every nonzero projection in $P$ is a supremum of a
finite sequence of mutually orthogonal atoms in $P$. Then, for all
$q\in P$, the infimum $e\wedge q$ exists in $E$.
\end{enumerate}
\end{corollary}

\begin{proof}

(i) The infimum $e\wedge p\sb{1}\sp{\,\perp}$ exists by Theorem \ref
{th:MGL3.8}. Similarly, as $e\wedge p\sb{1}\sp{\,\perp}\in E$, the
infimum $(e\wedge p\sb{1}\sp{\,\perp})\wedge p\sb{2}\sp{\,\perp}=
e\wedge(p\sb{1}\vee p\sb{2})\sp{\perp}$ exists in $E$. Continuing in
this way by induction, we obtain (i).

(ii) Obviously, $e\wedge 1=e$, so we can assume that $q\not=1$, whence
$q\sp{\perp}\not=0$. Therefore by hypothesis, there is a finite
sequence $p\sb{1}, p\sb{2}, ..., p\sb{n}$ of mutually orthogonal atoms
in $P$ such that $q\sp{\perp}=p\sb{1}\vee p\sb{2}\vee\cdots\vee p
\sb{n}$, and it follows from (i) that $e\wedge q$ exists in $E$.
\end{proof}

The synaptic algebra $A$ is said to be of \emph{rank} $r$, $r=1,2,3,...$
iff there are $r$, but not $r+1$ mutually orthogonal nonzero projections
in $P$. Clearly, a synaptic algebra of rank $r$ satisfies the hypothesis
of Corollary \ref{co:MG3.9} (ii). By \cite{FPSpin} and \cite[Corollary 4.4]
{FPSynap}, a positive-definite spin factor of dimension 2 or more is the
same thing as a synaptic algebra of rank 2. Therefore:

\begin{corollary}
If $A$ is a positive-definite spin factor of dimension 2 or more, $e\in E$,
and $q\in P$, then $e\wedge q$ exists in $E$.
\end{corollary}

\noindent We note that there are infinite-dimensional positive-definite spin
factors.

\end{document}